\documentclass[11pt]{article}

\usepackage{epsfig,epsf,fancybox}
\usepackage{amsmath}
\usepackage{mathrsfs}
\usepackage{amssymb}
\usepackage{graphicx}
\usepackage{color}
\usepackage{multirow}
\usepackage{paralist}
\usepackage{verbatim}
\usepackage{galois}
\usepackage{algorithm}
\usepackage{algorithmic}
\usepackage{boxedminipage}
\usepackage{booktabs}
\usepackage{accents}
\usepackage{stmaryrd}
\usepackage{subfig}

\newtheorem{theorem}{Theorem}[section]

\newtheorem{lemma}[theorem]{Lemma}

\newtheorem{definition}{Definition}[section]

\newtheorem{assumption}[theorem]{Assumption}

\,
\,

\newcommand{\x}{\mathbf x}
\newcommand{\y}{\mathbf y}
\newcommand{\s}{\mathbf s}
\newcommand{\sa}{\mathbf a}
\newcommand{\g}{\mathbf g}
\newcommand{\e}{\mathbf e}
\newcommand{\z}{\mathbf z}
\newcommand{\w}{\mathbf w}
\newcommand{\argmin}{\mathop{\rm argmin}}

\newcommand{\KCal}{\mathcal{K}}

\newcommand{\SCal}{\mathcal{S}}

\newcommand{\prox}{\textnormal{prox}}

\newcommand{\dom}{\textbf{dom}}

\newcommand{\br}{\mathbb{R}}

\newcommand{\ba}{\begin{array}}
\newcommand{\ea}{\end{array}}

\newcommand{\etal}{{\it et al.\ }}

\title{A Unified Scheme to Accelerate Adaptive  Cubic Regularization and Gradient Methods for Convex Optimization }

\author{
	Bo JIANG
	\thanks{Research Center for Management Science and Data Analytics, School of Information Management and Engineering, Shanghai University of Finance and Economics, Shanghai 200433, China. Email: isyebojiang@gmail.com. } \and
	Tianyi LIN
	\thanks{Department of Industrial Engineering and Operations Research, UC Berkeley, Berkeley, CA 94720, USA. Email: darren\_lin@berkeley.edu} \and
	Shuzhong ZHANG
	\thanks{Department of Industrial and Systems Engineering, University of Minnesota, Minneapolis, MN 55455, USA. Email: zhangs@umn.edu.}}

\begin{document}
	\maketitle
	
	\begin{abstract}
		In this paper, we propose a unified two-phase scheme to accelerate any high-order regularized tensor approximation approach on the smooth part of a composite convex optimization model. The proposed scheme has the advantage of not needing to assume any prior knowledge of the Lipschitz constants for the gradient, the Hessian and/or high-order derivatives.
		This is achieved by tuning the parameters used in the algorithm \textit{adaptively} in its process of progression, which has been successfully incorporated in high-order nonconvex optimization \cite{CartisGouldToint2018, Birgin-Gardenghi-Martinez-Santos-Toint-2017}. By adopting similar
		approximate measure of the subproblem in \cite{Birgin-Gardenghi-Martinez-Santos-Toint-2017} for {\it nonconvex optimization}, we establish the overall iteration complexity bounds for three specific algorithms to obtain an $\epsilon$-optimal solution for composite convex problems. In general, we show that the adaptive high-order method has an iteration bound of $O\left( 1 / \epsilon^{1/(p+1)} \right)$ if the first $p$-th order derivative information is used in the approximation, which has the same iteration complexity as in 
		\cite{Baes-2009, Nesterov-2018} where the Lipschitz constants are assumed to be known and the subproblems are assumed to be solved exactly. Thus, our results { partially address the problem of incorporating adaptive strategies into the high-order {\it accelerated} methods raised by Nesterov in \cite{Nesterov-2018}, although our strategies cannot assure the convexity of the auxiliary problem and such adaptive} strategies are already popular in high-order nonconvex optimization \cite{CartisGouldToint2018, Birgin-Gardenghi-Martinez-Santos-Toint-2017}.
		Specifically, we show that the gradient method achieves an iteration complexity in the order of $O\left( 1 / \epsilon^{1/2} \right)$, which is known to be best possible (cf.~\cite{Nesterov-2013-Introductory}), while the adaptive cubic regularization methods with the exact/inexact Hessian matrix both achieve an iteration complexity in the order of $O\left( 1 / \epsilon^{1/3} \right)$, which matches that of the original accelerated cubic regularization method presented in \cite{Nesterov-2008-Accelerating}.
		Our numerical experiment results show 
		a clear effect of 
		acceleration displayed in the adaptive Newton's method with cubic regularization
		on a set of regularized logistic regression instances. 
	\end{abstract}
	
	\vspace{0.25cm}
	
	
	\vspace{0.25cm}
	
	
\section{Introduction}	
In this paper, we consider the following generic composite convex optimization model:
\begin{equation}\label{prob:main}
F^*:=\min_{\x\in\br^d} \ F(\x) = f(\x) + r(\x),
\end{equation}
where $f: \br^d\rightarrow\br$ is \textit{convex} and \textit{smooth}, $r: \br^d\rightarrow\br$ is \textit{convex} but possibly \textit{nonsmooth} with  simple proximal mapping, and $F^*>-\infty$.
During the past 
decades, various classes of optimization algorithms for solving \eqref{prob:main} (especially when $r(\x) = 0$ and $F(\x)$ becomes smooth) have been developed and carefully analyzed; see \cite{Luenberger-1984-Linear, Nocedal-2006-Numerical, Nesterov-2013-Introductory} for relevant information and references therein. 
Despite nice theoretical property of the existing solution methods, there has been a practical concern regarding the implementation,
as many methods assume that some problem parameters such as the first and the second order Lipschitz constants are available, which may be hard to estimate in practice. 
It will be ideal to come up with optimization algorithms which automatically estimate such parametric values, making the algorithms easy-implementable while maintaining superior theoretical iteration bounds intact. In this case, we are demanding an algorithm to be less dependent on the knowledge of the problem structure at hand, therefore less prone to failures due to misinformation of such values. In this context, schemes that adaptively adjust the parameters used in the algorithms are often desirable, and
are likely leading to better numerical performances.
For instance, researchers tend to
train their deep learning models with adaptive gradient method (see e.g.\ AdaGrad in \cite{Duchi-2011-Adaptive})  due to its robustness and effectiveness (cf.\ \cite{Karparthy-2017-Peak}). In fact, Adam \cite{Kingma-2014-Adam} and RMSProp \cite{Tieleman-2012-Lecture} are recognized as the default solution methods in the deep learning setting. Among the category of  second order methods, Cartis {\em et al.}\/ \cite{Cartis-2011-Adaptive-I,Cartis-2011-Adaptive-II} proposed and analyzed an adaptive cubic regularized Newton's method, which soon became very popular due to its numerical efficiency. In a recent working paper~\cite{Nesterov-2018}, Nesterov proposed two implementable high-order methods where he also commented that an unsolved issue in his approach was a dynamic adjustment scheme for the Lipschitz constant for the highest derivative to achieve practical efficiency.

Another fundamental issue in optimization (as well as in machine learning) is to understand how the classical algorithms (including the first-order, second-order and high-order methods) can be accelerated. Nesterov  \cite{Nesterov-1983-Accelerated} put forward the very first accelerated (optimal in its iteration counts) gradient-based algorithm for smooth convex optimization. Beck and Teboulle \cite{Beck-2009-Fast} successfully extended Nesterov's approach to accomodate the problem in the form of \eqref{prob:main}.
Recently, accelerated algorithms were extended to incorporate second-order \cite{Nesterov-2008-Accelerating, Monteiro-2013-Accelerated} or high-order information \cite{Baes-2009, Nesterov-2018,Gasnikov-etal-2019} yielding a faster convergence rate.
However, these algorithms do require the knowledge of some problem-specific parametric (Lipschitz) constants.  

Overall, algorithms exhibiting both traits of \textbf{\textit{acceleration}} and \textbf{\textit{adaptation}} have been largely missing in the context of convex optimization.
{ To the best of our knowledge,
	besides this 
	and two other very recent reports \cite{Grapiglia-Nesterov-2019a,Grapiglia-Nesterov-2019b}},
there has been no other paper on accelerated second-order methods (or any high-order methods) that are fully independent of the problem constants while maintaining superior 
theoretical iteration bounds.
{ There are results on some combinations of the above flavors though.}
For instance, the adaptive cubic regularized Newton's method \cite{Cartis-2012-Evaluation} is Hessian-free and problem-parameter-free, and allows subproblem to be solved inexactly, but it merely achieves an iteration  bound of
$O\left( 1 /\epsilon^{1/2} \right)$ without acceleration. 
Thus, a natural question raises:
Can we develop an implementable accelerated second-order method with an iteration complexity lower than $O\left( 1 / \epsilon^{1/2} \right)$?
One goal of this paper is to present an affirmative answer to this question. It turns out the resulting accelerated adaptive cubic regularization algorithm displays an excellent numerical performance in solving a variety of large-scale machine learning models in our experiments.

\subsection{Related Work}\label{Section1:RelatedWork}
Nesterov's seminal work \cite{Nesterov-1983-Accelerated} triggered a burst of research on accelerating first-order methods.
There have been a good deal of recent efforts
studying the adaptive gradient methods with the optimal convergence rate
\cite{Duchi-2011-Adaptive, Nesterov-2013-Gradient, Lin-2014-Adaptive, Monteiro-2016-Adaptive}, and widely used in training the deep neural networks \cite{Kingma-2014-Adam, Tieleman-2012-Lecture}.
When the second-order information is available, Nesterov accelerated cubic regularized Newton's method \cite{Nesterov-2008-Accelerating} and {obtained} an improved iteration bound of $O\left( 1 / \epsilon^{1/3} \right)$. After that, Monteiro and Svaiter \cite{Monteiro-2013-Accelerated} managed to accelerate the Newton proximal extragradient method with a faster convergence rate of
$O\left( 1 / \epsilon^{2/7} \right)$. Very recently, Arjevani \etal\cite{Arjevani-Shamir-Shiff-2018} proved that $O\left( 1 / \epsilon^{2/7} \right)$ is actually a lower bound for the oracle complexity of the second-order methods for convex optimization, which implies Monteiro and Svaiter'method is an optimal second-order method.
In a recent work \cite{ghadimi2017second}, Ghadimi \etal generalized accelerated Newton's method with cubic regularization under inexact
second-order information. However, the complexity bound
is theoretically worse than that of its exact counterparts, and  only as good as that of the optimal first-order method. Baes \cite{Baes-2009} extended the method in \cite{Nesterov-2008-Accelerating} to the high-order case and further improved the iteration complexity to $O(1/\epsilon^{1/(p+1)})$.
Such extension was recently revisited by Nesterov
\cite{Nesterov-2018} elaborating 
on an efficient implementation when $p = 3$. On the other hand, Arjevani \etal\cite{Arjevani-Shamir-Shiff-2018} showed that the worst case iteration complexity of any high-order algorithm cannot be better than $O(1/\epsilon^{2/(3p+1)})$ and shortly after that an optimal high-order method was proposed in \cite{Gasnikov-etal-2019} with iteration complexity matching such lower bound. However, an additional bisection search is needed in each iteration of the method, and the number of bisection step consumed is bounded by a logarithmic factor in the given precision \cite{Monteiro-2013-Accelerated, Jiang-Wang-Zhang-2018, Bubeck-Jiang-Lee-Li-Sidford-2018}. Recently, Wilson~\etal\cite{Wilson-2019-Accelerating} proved that a family of first-order rescaled gradient descent algorithms can achieve the same convergence rate as the optimal $p$th tensor algorithms for optimizing the so-called $p$th order strongly smooth (see \cite{Wilson-2019-Accelerating}) objective.

Although the parameter-free approach is well studied in the first-order case, all the afore-mentioned high-order (including second-order) accelerated methods assume that the Lipshitz constant for certain degree of derivative is known, which may be unrealistic. To alleviate this, { Cartis \etal
	\cite{Cartis-2011-Adaptive-I,
		Cartis-2011-Adaptive-II,
		Cartis-2012-Evaluation}}
incorporated adaptive strategies into { the method of Nesterov and Polyak}  {\cite{Nesterov-Polyak-2006-Cubic}}, and further relaxed the criterion for solving each sub-problem while maintaining the convergence properties for both convex \cite{Cartis-2012-Evaluation} and non-convex \cite{Cartis-2011-Adaptive-I, Cartis-2011-Adaptive-II} cases. However, as mentioned earlier, the iteration complexity established in \cite{Cartis-2012-Evaluation} for convex optimization is merely $O\left( 1 / \epsilon^{1/2} \right)$.

In the context of nonconvex optimization, high-order information had already been proved to be useful to improve the convergence rate of the algorithms.
In particular, Birgin \etal\cite{Birgin-Gardenghi-Martinez-Santos-Toint-2017} first proposed a high-order regularization method similar to 
the cubic regularized algorithms in \cite{Cartis-2011-Adaptive-I, Cartis-2011-Adaptive-II}, 
using adaptive parameter-tuning and inexact subproblem solving. Interestingly, high-order information enables finding high-order critical points \cite{CartisGouldToint2018} yielding a solution with better quality, and the method in \cite{Birgin-Gardenghi-Martinez-Santos-Toint-2017}  was improved by Cartis \etal\cite{Cartis2017} to converge to second-order critical points. { Compared to the cubic regularized algorithm in \cite{Cartis-2011-Adaptive-I, Cartis-2011-Adaptive-II}, the high-order regularization methods in \cite{Birgin-Gardenghi-Martinez-Santos-Toint-2017, Cartis2017} have better iteration complexity for finding the first and the second order critical points.} Recently, the high-order methods were proposed to solve nonsmooth and/or constrained optimization problems \cite{Birgin-Gardenghi-Martinez-Santos-Toint-2016, MARTINEZ2017, CartisGouldToint2019, Chen-Toint-Wang-2019, Chen-Toint-2019, Bullins-Peng-2019}.
In the literature, there are second-order methods which are efficient for solving \eqref{prob:main}, and they are referred to as proximal (quasi-)Newton methods. The global convergence and the local superlinear rate of convergence of those methods have been shown in \cite{Lee-2014-Proximal} and more recently in \cite{Byrd-2016-Inexact}.
Grapiglia and Nesterov \cite{Grapiglia-Nesterov-2018} studied accelerated regularized Newton's methods of solving problem \eqref{prob:main}, where $f$ is twice differentiable with a H{\"o}lderian continuous Hessian, and they showed that the iteration bound depends on the H{\"o}lderian parameter. As we were finalizing this manuscript, we noticed that Grapiglia and Nesterov  \cite{Grapiglia-Nesterov-2019a, Grapiglia-Nesterov-2019b} extended their previous results to high-order case including an adaptive variant {with a similar theoretical guarantees to ours}, where the H{\"o}lderian parameter may be unknown. In comparison with the algorithm proposed in this paper, their algorithms only have one single phase, and are with a different acceptance condition, and use a different 
auxiliary function. {In addition, the adaptive parameter in their auxiliary function is updated via computing a positive solution of a suitable univariate polynomial equation, which guarantees a key inequality that ensures
	acceleration. Such a parameter in our algorithm is dynamically adjusted. Consequently, different parameter choices lead to slightly different numerical performances (see Section \ref{Subsec:l1-numerical} for more details).}




\subsection{Contributions}\label{Section1:Contribution}
The contributions of this paper can be summarized as follows.
We present a unified adaptive accelerating scheme that can be specialized to several optimization algorithms including gradient method, cubic regularized Newton's method with \textit{exact/inexact} Hessian and high-order method. For the gradient method, our adaptive algorthm achieves a convergence rate of $O\left( 1 / \epsilon^{1/2} \right)$ (Theorem \ref{Thm:first-order}) which matches the optimal rate for the first order methods \cite{Nesterov-2013-Introductory}. For the cubic regularized Newton's method we show that a global convergence rate
of $O\left( 1 / \epsilon^{1/3} \right)$ holds (Theorem \ref{Thm:second-order}) without assuming any knowledge of the problem parameters. We further prove that, even without the exact Hessian information, the same $O\left( 1 / \epsilon^{1/3} \right)$ rate of convergence (Theorem \ref{Thm:inexact-second-order}) is still achievable
for the cubic regularized approximative Newton's method. When our adaptive scheme reduces to the high-order method, the global rate of $O\left( 1 / \epsilon^{1/(p+1)} \right)$ is guaranteed by utilizing up to $p$-th order information, which achieves the same iteration bound as in Baes \cite{Baes-2009} {and Nesterov} \cite{Nesterov-2018}. Therefore, all the algorithms developed in this paper are problem-parameter-free due to the adopted {\it fully adaptive}\/ strategies, while retaining the same convergence rate.
Note that the accelerated first-order methods proposed in
\cite{Scheinberg-Goldfarb-Bai-2014, Calatroni-Chambolle-2017} shared the same characteristics, albeit their analysis is quite different. Similarly, the algorithms in \cite{Grapiglia-Nesterov-2019a, Grapiglia-Nesterov-2019b} by Grapiglia and Nesterov are parameter-independent, {and their convergence rates also match that of the nonadaptive ones.} In addition, the adaptivity enables an efficient implementation of the algorithm, while
numerical experiments are largely missing in the literature of high-order methods. There are a few numerical results reported in \cite{ghadimi2017second}, however the convergence rate is shown to be only as good as that of the accelerated first-order method.
In this paper, we performed numerous numerical experiments which showed a clear effect of acceleration of the proposed algorithms.
Finally, our convergence rate, which attains the same order of magnitude as those of \cite{Nesterov-2008-Accelerating, Baes-2009, Nesterov-2018}, is inferior than the rate for the optimal high-order method \cite{Gasnikov-etal-2019}. However, the gap between the two 
is small indeed; e.g.\ for $p=2$ the gap amounts to $O\left(1/\epsilon^{1/3 - 2/7}\right)= O\left(1/\epsilon^{1/21}\right)$. Arguably, the additional logarithmic factors required by the optimal method \cite{Jiang-Wang-Zhang-2018, Bubeck-Jiang-Lee-Li-Sidford-2018} could easily dominate the gap for practical $\epsilon$ values. 

\subsection{Notations and Organization} {\ }

{\em Notations.} We denote vectors by bold lower case letters, e.g., $\x$, and matrices by regular upper case letters, e.g., $X$. The transpose of a real vector $\x$ is denoted as $\x^\top$. For a vector $\x$, and a matrix $X$, $\left\|\x\right\|$ and $\left\|X\right\|$ denote the $\ell_2$ norm and the matrix spectral norm, respectively. { We use $\lambda_{\min}(X)$ to denote the minimum eigenvalue of the matrix $X$,} and $\nabla f(\x)$,  $\nabla^2 f(\x)$ and $\nabla^d f(\x)$ to indicate the gradient, the Hessian and $p$-th order derivative tensor of $f$ at $\x$, respectively. We denote
\begin{equation*}
\nabla^d f(\x)[\x^1,\dots,\x^d] := \sum_{i_1,\dots,i_d=1}^{n} \nabla^d f(\x)_{i_1,\dots,i_d} \x^{1}_{i_1} \dots \x^{d}_{i_d},
\end{equation*}
and $I$ denotes the identity matrix. For two symmetric matrices $A$ and $B$, $A \succeq B$ indicates that $A-B$ is symmetric positive semidefinite. The $\log(x)$ denotes the natural logarithm of $x$ for $x>0$.

{\em Organization.} The rest of the paper is organized as follows. In Section~\ref{section:preliminaries}, we introduce some preliminaries and the assumptions used throughout this paper. In Section~\ref{section:framework}, we propose our general framework to adaptively accelerate various optimization algorithms, and present the main theoretical results on the iteration complexity. Section~\ref{section:specializations} is devoted to specializations of our framework to first-order methods, second-order methods, and high-order methods.
In Section~\ref{section:experiment}, we present some preliminary numerical results on solving
$\ell_2$-regularzied and $\ell_1$-regularized logistic regression problems,
where acceleration of the method based on the adaptive cubic regularization for Newton's method is clearly observed. The details of all the proofs can be found in the appendix.

\section{Preliminaries}
\label{section:preliminaries}
Throughout this paper, we make the following assumptions for problem~\eqref{prob:main}.
\begin{assumption} \label{Assumption:Objective-Standard}
	$F$ is a proper, closed and convex function in the domain
	\[\dom(F) := \left\{\x\in\br^d\mid F(\x)<+\infty\right\}, \]
	and the optimal set of problem~\eqref{prob:main} is nonempty.
\end{assumption}
\begin{assumption} \label{Assumption:Gradient-Lipschitz-Continuous}
	The function $f$ is $p$-th continuously differentiable and $\nabla^j f$ is Lipschitz continuous with $L_j>0$ for $p-1 \leq j \leq p$, i.e.,
	\begin{equation}\label{Inequality:Gradient-Lipschitz-Continuous-Assumption}
	\left\|\nabla^j f(\x) - \nabla^j f(\y)\right\| \leq L_j\left\|\x - \y\right\|, \quad \forall\x,\y\in\dom(F)，
	\end{equation}
	where
	$$\left\|\nabla^j f(\x) - \nabla^j f(\y)\right\|=\max\limits_{\|z^i\|= 1,\,i=1,...,j}  \left(\nabla^j f(\x) - \nabla^j f(\y)\right)[z^{1} \dots z^{j}]
	$$
	is the operator norm associated with the tensor $\nabla^j f(\x) - \nabla^j f(\y)$.
\end{assumption}
We remark that {the $p$-th order Lipschitz continuity condition in} Assumption~\ref{Assumption:Gradient-Lipschitz-Continuous} is standard in the convergence analysis of $p$-th order optimization methods for minimizing smooth functions (\cite{Birgin-Gardenghi-Martinez-Santos-Toint-2017, Nesterov-2018}). {The Lipschitz continuous assumption on both $p$-th order and $(p-1)$-th order derivative is common in derivative-free method with $p=2$ \cite{Cartis-2012-Oracle} and is only needed in Subsection \ref{Subsection:Inexact-Hessian} of this paper to deal with second-order method with inexact Hessian information.}
We consider the following $p$-th order approximation of $f(\y)$ at point $\x$:
\begin{eqnarray*}
	&& \tilde{f}_p(\y;\x) \\
	&=& f(\x) + \left(\y-\x\right)^\top \nabla f(\x) + \frac{1}{2}\left(\y-\x\right)^\top \nabla^2 f(\x)\left(\y-\x\right) + \sum_{j=3}^{p} \frac{1}{j!} \nabla^j f(\x)\underbrace{\left[\y-\x, \ldots, \y-\x\right]}_{j \text{ terms}}.
\end{eqnarray*}
Under Assumptions~\ref{Assumption:Objective-Standard}-\ref{Assumption:Gradient-Lipschitz-Continuous}, the following two inequalities follow from residual analysis for the Taylor expansion (see also \cite{Birgin-Gardenghi-Martinez-Santos-Toint-2017, Nesterov-2018}):
\begin{eqnarray}\label{Inequality:Objective-Lipschitz-Continuous}
\left| f(\y) - \tilde{f}_p(\y;\x) \right|
\leq  \frac{L_p\left\|\y-\x\right\|^{p+1}}{(p+1)!},
\end{eqnarray}
and
\begin{equation}\label{Inequality:Gradient-Lipschitz-Continuous}
\left\| \nabla f(\y) - \nabla \tilde{f}_p(\y;\x) \right\| \leq \frac{L_p\left\|\y-\x\right\|^{p}}{p!}.
\end{equation}
Based on $\tilde{f}_p(\y;\x) $, we consider other approximations of $f(\y)$. We call function $\overline{m}(\y;\x)$  an \textit{effective approximation}  of the smooth function $f(\y)$ at point $\x$ if the following properties hold.
\begin{definition} \label{Definition:Approximate-Properties}
	We call $\overline{m}(\y; \x)$ to be an effective approximation of $f(\y)$ at a given point $\x \in \dom(F)$ if it satisfies the following three properties:
	\begin{itemize}
		\item[(i)] For any $\y \in \dom(F)$, it holds that
		\begin{equation} \label{Def:Effective-Objective-All}
		\left| f(\y) - \overline{m}(\y; \x) \right| \leq \bar{\kappa}_p \|\y - \x\|^{p} + \kappa_p \|\y - \x\|^{p+1}
		\end{equation}
		for some constants $\bar{\kappa}_p$ and $\kappa_p$.
		\item[(ii)] For any $\bar{\x} \approx \argmin_{\y \in \br^d} \ m(\y; \x, \sigma)$, it holds that
		\begin{align}
		\left| f(\bar{\x}) - \overline{m}(\bar{\x}; \x) \right| & \leq  \beta_p \|\bar{\x} - \x\|^{p+1}, \label{Def:Effective-Objective-Solution} \\
		\| \nabla f (\bar{\x}) - \nabla \overline{m}(\bar{\x}; \x) \| & \leq \rho_p \|\bar{\x} - \x\|^p; \label{Def:Effective-Gradient-Solution}
		\end{align}
		or the above two inequalities hold for {a pair of $\left(h, \bar \x \right)$ satisfying} $\| \bar{\x} - \x \| \ge h$ when $\overline{m}(\bullet\, ; \x)$ is additionally dependent on some positive number $h$, where all the parameters are constants.
		\item[(iii)] $\overline{m}(\y; \x)$ is convex in $\y$.
	\end{itemize}
\end{definition}
{We remark that the dependence of $\overline{m}(\bullet\, ; \x)$ on $h$ only occurs in Subsection \ref{Subsection:Inexact-Hessian}, where an approximated Hessian matrix is constructed based on step size $h$, leading to such dependence. The pair of $\left(h, \bar \x \right)$ satisfying \eqref{Def:Effective-Objective-Solution} and \eqref{Def:Effective-Gradient-Solution}
	for $\| \bar{\x} - \x \| \ge h$
	can be found by a procedure similar to 
	Algorithm 4.1 in~\cite{Cartis-2012-Oracle}.}
The specific choices of $\overline{m}(\y;\x)$ and the corresponding values of $\bar{\kappa}_p, {\kappa}_p, \beta_p, \rho_p$  will be discussed in Section \ref{section:specializations} and summarized in Table \ref{tab:m-choice}.  { With an effective approximation $\overline{m}(\y; \x)$ of $f(\y)$ in hand,  the approximation model for the objective function $F(\y)$ is now given by
	\begin{equation}\label{Subprob:p-power}
	m(\y; \x, \sigma) := \overline{m}(\y; \x)+ \frac{\sigma \left\|\y - \x\right\|^{p+1}}{p+1} + r\left(\y\right).
	\end{equation}
}


We end this section by specifying the definitions of \textit{$\varepsilon$-optimality} and \textit{proximal mapping} which are frequently used in this paper.
\begin{definition}[$\varepsilon$-optimality]
	Given $\varepsilon\in\left(0,1\right)$, $\x\in\br^d$ is said to be $\varepsilon$-optimal to problem~\eqref{prob:main} if
	\[ F(\x) - F(\x^*) \leq \varepsilon, \]
	where $\x^*\in\br^d$ is an optimal solution to problem~\eqref{prob:main}.
\end{definition}
\begin{definition}[proximal mapping]
	The proximal mapping of $r$ at $\x \in \br^d$ is
	\[ \prox_r(\x) := \argmin_{\z\in\br^d} \ r(\z) + \frac{\left\|\z - \x\right\|^2}{2}. \]
\end{definition}

\section{Algorithmic Framework}\label{section:framework}
In this section, we propose a unified framework for accelerating the adaptive methods. This framework is composed of two subroutines: \textsf{Simple Adaptive Subroutine (SAS)} and \textsf{Accelerated Adaptive Subroutine (AAS)}. Specifically, the framework starts with \textsf{SAS}, which terminates as soon as one successful iteration is identified. Then, the output of \textsf{SAS} is used as the initial point to run \textsf{AAS} until a sufficient number of successful iterations $T_2$ are observed. {We also adopt the same auxiliary model as that used by Nesterov in \cite{Nesterov-2008-Accelerating,Nesterov-2018} except for the appearance of the subgradient $\xi$ due to the additional nonsmooth regularization:
	\begin{equation}\label{auxiliary-function}
	\psi_{j+1}(\z, \tau_{j+1}) = l_{j+1}(\z) + \tau_{j+1}R(\z)
	\end{equation}
	where $R(\z) = \frac{1}{2(p+1)}\left\|\z - \bar{\x}_0\right\|^{p+1}$, $l_0(\z)  =  F(\bar{\x}_0)$, $l_{j+1}(\z) = l_j(\z) + \Delta l_j\left(\z; \bar{\x}_{j+1}, \bar{\xi}_{j+1}\right)$, and
	\begin{eqnarray*}
		\Delta l_j(\z, \x, \xi) & = & \frac{\Pi_{\ell =2}^{{p+1}}(j+\ell)}{p!} \left[ F(\x) +\left(\z - \x\right)^\top\left(\nabla f(\x) + \xi\right)\right].
\end{eqnarray*}}
The details of our algorithmic framework are summarized in Algorithm~\ref{Algorithm:Framework} (in the order of ``Main Procedure'', ``SAS'' and ``AAS'').

\begin{algorithm}[!t]\small
	\caption{A Generic Unified Adaptive Acceleration Framework (UAA)} \label{Algorithm:Framework}
	{\bf Main Procedure:}
	\begin{algorithmic}
		\STATE \textbf{Input:} $\x_0\in\br^d$, $\sigma_0 \geq \sigma_{\min} > 0$, $\tau_0 > 0$, $\gamma_2 > \gamma_1 > 1$, $\gamma_3 > 1$, $\eta > 0$, and approximate model $m(\cdot)$.
		\STATE Phase I (SAS): $\left[ {\bar{\x}_0, \sigma^{SAS}}\right] = \textsf{SAS}\left(\x_0, \sigma_0, \sigma_{\min}, \gamma_1, \gamma_2, m \right)$.
		{ \IF{$p \ge 3$ ($p$ is the power index of the regularizer in $m(\cdot)$)}
			\STATE \quad  $\sigma_0^{AAS}=\max\{\sigma^{SAS},  -\lambda_{\min}\left(\nabla^2\overline{m}(\bar{\x}_0;{{\x_0}}) \right)/\|\bar{\x}_0 - {\x_0}\|^{p-1}\}$.
			\ELSE
			\STATE   $\sigma_0^{AAS} = \sigma^{SAS}$.
			\ENDIF
		}
		\STATE  Phase II (AAS): $\left[\x_{out}\right] = \textsf{AAS}\left({\bar{\x}_0}, \sigma_0^{AAS}, \sigma_{\min}, \tau_0, \gamma_1, \gamma_2, \gamma_3, \eta, m \right)$.
		\STATE \textbf{Output:} an $\varepsilon$-optimal solution $\x_{out}$.
	\end{algorithmic}
	\vspace{0.3cm}
	
	{\bf Simple Adaptive Subrutine:} $\textsf{SAS}\left(\x_0, \sigma_0, \sigma_{\min}, \gamma_1, \gamma_2, m\right)$ \\
	\begin{algorithmic}
		\STATE \textbf{Initialization:} the total iteration count $i=0$ and successful iteration count $j=0$.
		\REPEAT
		\STATE compute $\x_{i+1} \approx \argmin_{\x \in \br^d} \ m( \x; \x_i, \sigma_i)$.
		\IF{$F(\x_{i+1}) - m\left(\x_{i+1}; \x_i, \sigma_i\right) < 0$}
		\STATE update $\sigma_{i+1} \in \left[\sigma_{\min}, \sigma_i\right]$ and $j=j+1$.
		\ELSE
		\STATE update $\x_{i+1} = \x_i$ and $\sigma_{i+1} \in \left[\gamma_1\sigma_i, \gamma_2\sigma_i\right]$.
		\ENDIF
		\STATE update $i=i+1$.
		\UNTIL{the successful iteration count $j=1$.}
		\STATE \textbf{Output:} the total iteration number $i$, the iterate $\x_i$ and the regularization parameter $\sigma_i$.
	\end{algorithmic}
	\vspace{0.3cm}
	
	{\bf Accelerated Adaptive Subroutine:} $\textsf{AAS}\left(\x_0, \sigma_0, \sigma_{\min}, \tau_0, \gamma_1, \gamma_2, \gamma_3, \eta, m \right)$ \\
	\begin{algorithmic}
		\STATE \textbf{Initialization:} the total iteration count $i=0$ and successful iteration count $j=0$.
		\STATE \textbf{Initial Step:} construct the auxiliary model $\psi_0(\z, \tau_0)=l_0(\z)+\tau_0 R(\z)$, update $\bar{\x}_0 = \x_0$, compute $\z_0 = \argmin_{\z\in\br^d} \ \psi_0(\z, \tau_0)$ and $\y_0 = \frac{1}{p+2}\bar{\x}_0 +  \frac{p+1}{p+2}\z_0$.
		\FOR{$i=0,1,2,\ldots $ {until convergence,}}
		\STATE compute $\x_{i+1} \approx \argmin_{\x \in \br^d} \ m(\x; \y_j, \sigma_i)$ and $\xi_{i+1}\in\partial r\left(\x_{i+1}\right)$.
		\IF{$\theta(\x_{i+1}, \y_j, \xi_{i+1}) \geq \eta$}
		\STATE update $\bar{\x}_{j+1} = \x_{i+1}$ and $\bar{\xi}_{j+1} = \xi_{i+1}$.
		\STATE update $l_{j+1}(\z) = l_j(\z) + \Delta l_j\left(\z; \bar{\x}_{j+1}, \bar{\xi}_{j+1}\right)$ and $\tau_{j+1} = \tau_j$.
		\REPEAT
		\STATE update $\tau_{j+1} = \gamma_3\tau_{j+1}$, and\\
		$\z_{j+1} = \argmin_{\z\in\br^d} \ \left\{ \psi_{j+1}(\z, \tau_{j+1}) = l_{j+1}(\z) + \tau_{j+1} R(\z)\right\}$.
		\UNTIL{$\psi_{j+1}(\z_{j+1}, \tau_{j+1}) \geq \frac{\Pi_{\ell=1}^{p+1}(j+1+ \ell)}{(p+1)!} F(\bar{\x}_{j+1})$}
		\STATE update $\y_{j+1} = \frac{(j+1)+1}{(j+1)+p+2} \bar{\x}_{j+1} + \frac{p+1}{(j+1)+p+2}\z_{j+1}$, $\sigma_{i+1} \in \left[\sigma_{\min}, \sigma_i\right]$ and $j=j+1$.
		\ELSE
		\STATE update $\x_{i+1} = \x_i$ and $\sigma_{i+1} \in \left[\gamma_1\sigma_i, \gamma_2\sigma_i\right]$.
		\ENDIF
		\ENDFOR
		\STATE \textbf{Output:} the total number of iterations $i$ and the iterate $\x_i$.
	\end{algorithmic}
\end{algorithm}
We remark that the
two-phase scheme is necessary in our analysis to establish the accelerated rate of convergence { while maintaining promising numerical performance.}
Some key ingredients of the framework are explained below:

\textbf{Input:} The input contains nine elements: $\x_0\in\br^d$ is the initial point; $\sigma_0$ is the initial regularization parameter for the approximate model; $\sigma_{\min}$ is the safeguard level for the regularization parameter; $\tau_0$ is the initial regularization parameter for the auxiliary model; $\gamma_1, \gamma_2, \gamma_3 \in \left(1, +\infty\right)$ are the ratios for adapting $\sigma$ and $\tau$, $\eta>0$ is the threshold for \textsf{AAS}. { The approximation model $m(\cdot)$ is as given in \eqref{Subprob:p-power}.}

In each iteration of our algorithm, we seek an approximate solution of minimizing $m(\cdot; \x, \sigma)$, which is defined as follows.


\begin{definition} \label{Definition:Approximate-Model}
	Let us call
	$\bar{\x} \approx \argmin_{\y \in \br^d} \ m(\y; \x, \sigma)$ with $\bar{\xi} \in\partial r\left(\bar{\x}\right)$ if \\ $m(\bar \x; \x, \sigma) \le  m(\x; \x, \sigma)$ and
	\begin{equation}\label{Criterion:Approximate-Adaptive}
	\left\| \nabla \overline{m}(\bar{\x}; \x)+ \sigma \left\| \bar{\x} - \x\right\|^{p-1} \left(\bar{\x} - \x\right) + \bar{\xi}\right\| \leq \kappa_\theta \left\|\bar{\x} - \x\right\|^p, \qquad \kappa_\theta>0.
	\end{equation}
\end{definition}
Note that the condition \eqref{Criterion:Approximate-Adaptive} without $\bar{\xi}$ was firstly proposed in \cite{Birgin-Gardenghi-Martinez-Santos-Toint-2017} for smooth nonconvex optimization. In the case of $r(\x) =0$ and $p=2$, such approximativeness measure does not include the following condition:
\begin{equation}\label{Subprob:Cubic}
(\bar \x- \x )^\top \nabla f(\x) + (\bar \x- \x )^\top \nabla^2 f(\x) (\bar \x- \x ) + \sigma\left\|(\bar \x- \x )\right\|^3 = 0,
\end{equation}
and thus weaker than the one used in \cite{Cartis-2011-Adaptive-I}.
This relaxation also suggests other approximations and implementable solution methods for \eqref{Subprob:Cubic}. For instance, Carmon and Duchi proposed to use gradient descent method to solve \eqref{Subprob:Cubic}, and they proved that it works well even when $m(\y; \x, \sigma)$ is nonconvex. However, the function $m(\y; \x, \sigma)$ in our case is {strictly convex as long as $y \neq x$}, and thus the gradient descent is {likely} to exhibit a fast (linear) convergence behavior. {When $r(\x) \neq0$ and $p=2$, we solve the subproblem with accelerated proximal gradient method (APGD) as $m(\y; \x, \sigma)$ is guaranteed to be convex in this case. For more general case of $r(\x) \neq0$ and $p\ge 3$,
	we may resort to some existing algorithms \cite{Ghadimi-Lan-Zhang-2016, Ghadimi-Lan-2016, Jiang-Lin-Ma-Zhang-2019} tailored for nonconvex composite optimization. In particular, we adopt
	the proximal gradient method (PGD) \cite{Ghadimi-Lan-Zhang-2016}} with the initialization $\x_0=\x$ and the step size $\alpha>0$ with the $k$-th iteration being:
\[ \x_{i,k+1} = \prox_{r/\alpha}\left(\x_{i,k} - \frac{\nabla \overline{m}(\x_{i,k}; \x)+ \sigma \left\| \x_{i,k} - \x\right\|^{p-1} \left(\x_{i,k} - \x\right)}{\alpha}\right), \]
until $\x_{i,k}\approx \argmin_{\y \in \br^d} \ m(\y; \x, \sigma)$.

{\textbf{Solving auxiliary model}: In this framework, we update $\z_{j+1}$ by solving the auxiliary problem as defined in \eqref{auxiliary-function}:
	$\z_{j+1} = \argmin_{\z\in\br^d} \ \psi_{j+1}(\z, \tau_{j+1})$,
	where the parameter $\tau_{j+1}$ is tuned dynamically in the algorithm.}
This function is the bridge for the two-sided inquality in \eqref{KeyInquality:F} to establish the iteration bound.
In fact, the above subproblem can be solved exactly. To see this, write out
the optimality condition and get:
\[ \nabla l_{j+1}(\z_{j+1}) + \frac{\tau_{j+1}\left\|\z_{j+1} - \bar{\x}_0\right\|^{p-1}\left(\z_{j+1} - \bar{\x}_0\right)}{2} =0, \]
which implies that
\[ \left\|\z_{j+1} - \bar{\x}_0\right\| = \left( \frac{2 \|\nabla\ell_{j+1}(\z_{j+1}) \| }{\tau_{j+1}} \right)^{1/p}. \]
Moreover, we observe that $l_{j+1}(\z)$ is a linear function of $\z$ and hence $\nabla\ell_{j+1}(\z_{j+1})$ is independent of $\z_{j+1}$. Consequently, we conclude that
\[ \z_{j+1} = \bar{\x}_0 - \left(\frac{2}{\tau_{j+1}}\right)^{1/p}\frac{\nabla l_{j+1}(\z_{j+1})}{\left\|\nabla l_{j+1}(\z_{j+1})\right\|^{1-1/p}}. \]
\textbf{Criterion:} The criterion for determining the successful iteration in \textsf{AAS} is
\[\theta(\x_{i+1}, \y_j, \xi_{i+1}) \geq \eta.\]
In particular, for $p \geq 1$ we define $\theta(\x, \y, \xi)$ as
\[ \theta(\x, \y, \xi) = \frac{\left(\y - \x\right)^\top\left(\nabla f(\x) + \xi\right)}{\left\| \y - \x \right\|^{p+1}}. \]
\textbf{Output:} The output contains the total number of iterations $i$ and the iterate $\x_i$. Note that $\x_i$ is an $\varepsilon$-optimal solution for  problem~\eqref{prob:main}.

\subsection{Iteration Complexity of the UAA}
In this subsection, we first make the following assumption.

\begin{assumption}\label{Asu:bounded-levelset}
	Suppose $\x_0$ is the starting point of our algorithm and $\x^*$ is an optimal solution of problem \eqref{prob:main}. The level set $\mathcal{L}(x_0,\sigma) := \{ x \in \br^d \;  | \; m(\x; \x_0, \sigma) \le m(\x_0; \x_0, \sigma) = F(\x_0) \}$ of $m(\cdot)$ at $\x_0$ with regularization parameter $\sigma$ is bounded when $\sigma = {\sigma}_{\min}$, and that
	\begin{equation}\label{Bounded-levelset}
	\max_{\x \in \mathcal{L}(x_0, {\sigma}_{\min})} \| \x - \x^*  \| \le D < \infty.
	\end{equation}
\end{assumption}

Then we present the main theoretical results on the iteration complexity of \textsf{UAA}.
\begin{theorem}\label{Thm:iteration-complexity-uaa} Let the sequence of iterates $\{\bar{\x}_j, \ j \geq 0\}$ be generated by \textsf{AAS} in \textsf{UAA} and $\x^*$ be an optimal solution for \eqref{prob:main}. Denote
	\[ C := (p+1)! \left( \frac{2(p+1)\kappa_p + 2{\hat{\sigma}_1} + {\hat \sigma_2}}{2(p+1)}D^{p+1} + \, \bar{\kappa}_p D^p +  (\kappa_\theta + {\hat{\sigma}_1} ) (2D)^{p+1} \right), \]
	{where $\hat \sigma_1 :=\max\left\{\bar{\sigma}_1, \frac{L_p}{(p-1)!} \right\} $ and $\hat \sigma_2 :=\max\left\{ \tau_0, \frac{2^p\gamma_3\left(\rho_p + \bar{\sigma}_2 + \kappa_{\theta}\right)^{p+1} p^{p-1}}{\eta^p (p-1)!}\right\} $.}
	Then it holds that
	\[
	F(\bar{\x}_j) - F(\x^*) \leq \frac{C }{\Pi_{\ell=1}^{p+1}(j+\ell)},\]
	which implies that the total iteration number required to reach $\varepsilon$-optimal solution can be bounded by
	\begin{eqnarray*}
		j &\leq & 2 + \frac{2}{\log(\gamma_1)}\log\left(\frac{\bar{\sigma}_1}{\sigma_{\min}}\right) + \left\lceil \frac{1}{\log\left(\gamma_3\right)} \log\left(\frac{{2^p}\left(\rho_p + \bar{\sigma}_2 + \kappa_{\theta}\right)^{p+1} p^{p-1}}{\eta^p (p-1)!\tau_0}\right) \right\rceil  \\
		& & + \, \left(1+\frac{2}{\log(\gamma_1)}\log\left(\frac{\bar{\sigma}_2}{\sigma_{\min}}\right)\right)\left[1 + \left(\frac{C}{\varepsilon}\right)^{\frac{1}{p+1}} \right].
	\end{eqnarray*}
\end{theorem}
The proof of the theorem is technically involved and hence postponed to the appendix. 
%

\section{Specializations of the UAA}\label{section:specializations}
In this section, we provide some concrete choices of $\overline m (\y; \x)$, which leads to different iteration complexities of the correspoding algorithms. 
{To present a holistic picture of the results in this section, we summarize the forms of $\overline m (\y; \x)$ associated with different settings in Table \ref{tab:m-choice}.
	
	\begin{table}		\caption{Specific choices of $\overline{m}(\y;\x)$}\label{tab:m-choice}\vspace*{-2.5em}
				\vspace{0.2cm}
		\begin{center}
		\begin{tabular}{|c|c|c|c|c|c|} \hline
			{\footnotesize Derivative Inf.} & $\overline{m}(\y;\x)$ & $\bar{\kappa}_p$&${\kappa_p}$ & $\beta_p$ & $\rho_p$ \\ \hline \hline
			{\footnotesize up to $1$st-order } &$\tilde f_1(\y;\x)$ &$0$&$\frac{L_1}{2}$&$\frac{L_1}{2}$&$L_1$ \\ \hline
			{\footnotesize up to $2$nd-order } &$\tilde f_2(\y;\x)$ &$0$&$\frac{L_2}{6}$&$\frac{L_2}{6}$&$\frac{L_2}{2}$ \\ \hline
			{\footnotesize inexact Hessian } &{\footnotesize $\tilde f_1(\y;\x) + \frac{1}{2} (\y - \x)^{\top} H(\x) (\y - \x)$} &{\footnotesize $L_1+\kappa$}&$\frac{L_2}{6}$&$\frac{L_2 + 3 \kappa}{6}$&$\frac{L_2 + 2 \kappa}{2}$ \\  \hline
			{\footnotesize up to $p$th-order } &$\tilde{f}_p(\y;\x)$ &$0$&$\frac{L_p}{(p+1)!}$&$\frac{L_p}{(p+1)!}$&$\frac{L_p}{p!}$ \\
			\hline
		\end{tabular}
	\end{center}
\end{table}}

\subsection{First-Order Adaptive Accelerating Method}
The most popular choice of $\overline{m}(\y; \x)$ is the first order approximation:
$$
\overline{m}(\y; \x) = \tilde f_1(\y;\x) = f(\x) + (\y - \x)^{\top} \nabla f(\x) .
$$
Obviously, it is convex, and
by \eqref{Inequality:Objective-Lipschitz-Continuous} and \eqref{Inequality:Gradient-Lipschitz-Continuous},  (i) and (ii) in Definition  \ref{Def:Effective-Objective-All} are satisfied with
$$
\kappa_1 = \beta_1 = \frac{L_1}{2}, \quad \rho_1 = L_1, \quad \bar \kappa_1=0.
$$
Moreover, the subproblem becomes $\min_{\y \in \br^d} f(\x) + (\y - \x)^T\nabla f(\x) + \frac{\sigma\| \y - \x \|^2}{2} + r(\y)$, which has a closed form solution since $r(\cdot)$ has an easy proximal mapping.
Therefore, $\kappa_\theta = 0$ and we have the following iteration bound.
\begin{theorem}\label{Thm:first-order} Letting $\overline{m}(\y; \x) = \tilde f_1(\y;\x)$ in \textsf{UAA}, we obtain an adaptive accelerating first-order method, and  the total iteration number of getting an $\epsilon$-optimal solution is
	\begin{eqnarray*}
		& & 2 + \frac{2}{\log(\gamma_1)}\log\left(\frac{\bar{\sigma}_1}{\sigma_{\min}}\right) + \left\lceil \frac{1}{\log\left(\gamma_3\right)} \log\left(\frac{2\left(\frac{L_1}{2} + \bar{\sigma}_2 \right)^{2} }{\eta \tau_0}\right) \right\rceil \\
		& & + \left(1+\frac{2}{\log(\gamma_1)}\log\left(\frac{\bar{\sigma}_2}{\sigma_{\min}}\right)\right)\left[1 + \left(\frac{C_1}{\varepsilon}\right)^{\frac{1}{2}} \right]
	\end{eqnarray*}
	where $C_1 =   \frac{2 L_1 + 2 \bar{\sigma}_1 + \tau_0}{2}D^{2}$.
\end{theorem}

\subsection{Second-Order Adaptive Accelerating Method}
\subsubsection{Exact Hessian Approximation}
The second order approximation of $f$ under exact Hessian is given by
$$
\overline{m} (\y; \x) = \tilde f_2(\y;\x) = f(\x) + (\y - \x)^{\top} \nabla f(\x) + \frac{1}{2} (\y - \x)^{\top} \nabla^2 f(\x) (\y - \x).
$$
It is still a convex function. Moreover, by \eqref{Inequality:Objective-Lipschitz-Continuous} and \eqref{Inequality:Gradient-Lipschitz-Continuous},  (i) and (ii) in Definition  \ref{Def:Effective-Objective-All} are satisfied with
$$
\kappa_2 = \beta_2 = \frac{L_2}{6}, \quad \rho_2 = \frac{L_2}{2}, \quad \bar \kappa_2=0.
$$
Moreover, since $\nabla \overline{m} (\y; \x) = \nabla^2 f(\x) \succeq 0$, $\overline{m} (\y; \x)$ is a convex function.
Therefore, we have the following iteration bound.
\begin{theorem}\label{Thm:second-order} Letting $\overline{m}(\y; \x) = \tilde f_2(\y;\x)$ in \textsf{UAA}, we obtain an adaptive accelerating cubic regularized Newton's method, and  the total iteration number of getting an $\epsilon$-optimal solution is
	\begin{eqnarray*}
		& & 2 + \frac{2}{\log(\gamma_1)}\log\left(\frac{\bar{\sigma}_1}{\sigma_{\min}}\right) + \left\lceil \frac{1}{\log\left(\gamma_3\right)} \log\left(\frac{4\left(\frac{L_2}{6} + \bar{\sigma}_2 + \kappa_{\theta}\right)^{3} }{\eta^2 \tau_0}\right) \right\rceil \\
		& & + \left(1+\frac{2}{\log(\gamma_1)}\log\left(\frac{\bar{\sigma}_2}{\sigma_{\min}}\right)\right)\left[1 + \left(\frac{C_2}{\varepsilon}\right)^{\frac{1}{3}} \right]
	\end{eqnarray*}
	where $C_2 = 6 \left( \frac{{L_2} + 2 \bar{\sigma}_1 + \tau_0 }{6}D^{3} + 8 \kappa_\theta D^3 \right)$.
\end{theorem}

\subsubsection{Inexact Hessian Approximation}\label{Subsection:Inexact-Hessian}
We study the scenario where the Hessian information is possibly not available; instead, we can construct an approximation of the Hessian $\nabla^2 f(\x_i)$ by first computing $d$ forward gradient differences at $\x_i$ with a step size $h_i\in\br$,
\[ A_i = \left[\frac{\nabla f(\x_i + h_i \e_1) - \nabla f(\x_i)}{h_i}, \ldots, \frac{\nabla f(\x_i + h_i \e_d) - \nabla f(\x_i)}{h_i}\right], \]
and symmetrizing the resulting matrix: $\widehat{H}(\x_i)=\frac{1}{2}\left(A_i + A_i^\top\right)$ and then further adding a sufficiently large constant multiple of identity matrix to $\widehat{H}(\x_i)$: $H(\x_i) = \widehat{H}(\x_i) + \kappa_c h_i I$, where $\e_j$ is the $j$-th vector of the canonical basis. It is well known in {Section 7.1 of~\cite{Nocedal-2006-Numerical}} that, for some constant $\kappa_e>0$,  we have
\[\left\|\widehat{H}(\x_i) - \nabla^2 f(\x_i)\right\| \leq \kappa_e h_i.\]
Consequently, it holds that
\[ \left\|H(\x_i) - \nabla^2 f(\x_i)\right\| \leq \left(\kappa_e + \kappa_c\right)h_i. \]
That is to say, the gap between exact and inexact Hessian can be bounded by a multiple of the step size $h_i$. This together with Algorithm 4.1 in~\cite{Cartis-2012-Oracle} motivates a procedure for searching a pair of $\left(h_i, \x_{i+1}\right)$
such that
{
	\begin{equation}\label{kappa-hs}
	h_i \leq \min\{\kappa_{hs}, \kappa_{hs}\left\|\x_{i+1} - \x_i\right\|  \} \quad \mbox{for some} \quad \kappa_{hs}>0.
	\end{equation} 	
	This procedure is adapted from Algorithm 2 in the first version of this paper (\cite{Jiang-2017-Unified}) by replacing the early stop criterion $\|\nabla f(\x_{i+1})\| \le \epsilon$ by $\|\nabla f(\x_{i+1}) + \xi \| \le \epsilon$ with ${\xi} \in\partial r\left(\x_{i+1}\right)$ as we consider composite optimization in this paper. Similar to Lemma 4.1 in \cite{Jiang-2017-Unified}, we can show that one call of this procedure
	requires $O(\log(1/\epsilon))$ number of iterations with $n$ additional gradient computations in each iteration.	Since this procedure is needed in both successful iteration and unsuccessful iteration in the main loop, it will add a logarithmic factor of $1/\epsilon$ to the overall iteration complexity of the method.} Letting $\kappa=\left(\kappa_e + \kappa_c\right) \kappa_{hs}$, we conclude that
\begin{equation}\label{Criterion:inexact-Hessian-stepsize}
\left\|H(\x_i) - \nabla^2 f(\x_i)\right\| \leq \kappa\left\|\x_{i+1} - \x_i\right\|.
\end{equation}
Therefore, we set
$$
\overline{m} (\y; \x) = f(\x) + (\y - \x)^{\top} \nabla f(\x) + \frac{1}{2} (\y - \x)^{\top} H(\x) (\y - \x)
$$
as the second order approximation of $g$ under the inexact Hessian. It follows from \eqref{Inequality:Objective-Lipschitz-Continuous}, \eqref{Inequality:Gradient-Lipschitz-Continuous} and \eqref{Criterion:inexact-Hessian-stepsize} that
\begin{eqnarray*}
	& & \left| f({\x}_{i+1}) - \overline{m}({\x}_{i+1}; \x_i) \right|\\
	& \leq & \left|\frac{1}{2} (\x_{i+1} - \x_i)^{\top} \nabla^2 f(\x_i)(\x_{i+1} - \x_i) -\frac{1}{2} (\x_{i+1} - \x_i)^{\top} H(\x_i)(\x_{i+1} - \x_i) \right| \\
	& & + \left| f({\x}_{i+1}) - f_2(\x_{i+1};\x_i) \right| \\
	& \leq &  \frac{L_2}{6} \|{\x}_{i+1} - \x_i\|^{3} + \frac{\kappa}{2}\|{\x}_{i+1} - \x_i\|^{3} = \frac{L_2 + 3 \kappa}{6} \|{\x}_{i+1} - \x_i\|^{3},
\end{eqnarray*}
and
\begin{eqnarray*}
	&&\left| \nabla f ({\x}_{i+1}) - \nabla \overline{m}({\x_{i+1}}; \x_i)  \right|\\
	&\leq & \left|\nabla f({\x}_{i+1}) -\nabla f_2(\x_{i+1};\x_i) \right| + \left|       \nabla^2 f(\x_i)(\x_{i+1} - \x_i) - H(\x_i)(\x_{i+1} - \x_i) \right| \\
	&\leq&  \frac{L_2}{2} \|{\x}_{i+1} - \x_i\|^{2} + {\kappa}\|{\x}_{i+1} - \x_i\|^{2} = \frac{L_2 + 2 \kappa}{2} \|{\x}_{i+1} - \x_i\|^{2}.
\end{eqnarray*}
Moreover, since $\nabla f(\x)$ is Lipschitz continuous with $L_{1}>0$, we have
\begin{equation*} \label{Inequality:Hessian-Boundedness}
\left\|\nabla^2 f(\x)\right\| \leq L_{1}, \quad \x\in\dom(F),
\end{equation*}
{and thus $\| \widehat{H}(\x) \| \le L_1$, which further implies $\| H(\x) \| \le L_1+ \kappa_c h_i \le L_1+ \kappa_c \kappa_{hs} \le L_1+ \kappa$.}
As a result
\begin{eqnarray*}
	&&\left| f(\y) - \overline{m}(\y; \x) \right|\\
	&\leq & \left| f({\y}) - f_2(\y;\x) \right| + \left|     \frac{1}{2} (\y - \x)^{\top} \nabla^2 f(\x)(\y - \x) -\frac{1}{2} (\y - \x)^{\top} H(\x)(\y - \x) \right| \\
	&\leq&  \frac{L_2}{6} \| \y - \x\|^{3} + \frac{1}{2}\| \y - \x \|^{2} \left(\|\nabla^2 f(\x) \| + \|H(\x) \| \right) \\
	&\leq & \frac{L_2}{6} \| \y - \x\|^{3} + {(L_1 + \kappa)} \| \y - \x\|^{2}.
\end{eqnarray*}
Finally, since $f$ is convex and $\kappa_c$ is sufficiently large such that $\kappa_c\geq\kappa_e$, we have
\[ H(\x_i) \succeq \nabla^2 f(\x_i) - \kappa_e h_i I + \kappa_c h_i I \succeq 0, \]
and $\overline{m}(\y; \x_i)$ is convex as well.
Therefore, all three conditions in  Definition  \ref{Def:Effective-Objective-All} are satisfied with
$$
\beta_2 = \frac{L_2 + 3 \kappa}{6},\quad \quad \rho_2 = \frac{L_2 + 2 \kappa}{2},\quad \kappa_2 = \frac{L_2}{6},\quad \bar \kappa_2={L_1 + \kappa}.
$$
Therefore, we have the following iteration bound.
\begin{theorem}\label{Thm:inexact-second-order} Letting $\overline{m}(\y; \x) = f(\x) + (\y - \x)^{\top} \nabla f(\x) + \frac{1}{2} (\y - \x)^{\top} H(\x) (\y - \x)$ in \textsf{UAA}, we obtain an adaptive accelerating cubic regularized approximate Newton's method, and  the total iteration number of getting an $\epsilon$-optimal solution is
	\begin{eqnarray*}
		& & 2 + \frac{2 \log\left(\frac{\bar{\sigma}_1}{\sigma_{\min}}\right)}{\log(\gamma_1)} + \left\lceil \frac{1}{\log\left(\gamma_3\right)} \log\left(\frac{4\left(\frac{L_2 + 2 \kappa}{2} + \bar{\sigma}_2 + \kappa_{\theta}\right)^{3} }{\eta^2\tau_0}\right) \right\rceil  \\
		& & + \left(1+\frac{2}{\log(\gamma_1)}\log\left(\frac{\bar{\sigma}_2}{\sigma_{\min}}\right)\right)\left[1 + \left(\frac{\bar C_2}{\varepsilon}\right)^{\frac{1}{3}} \right],
	\end{eqnarray*}
	where $\bar C_2 = 6 \left( \frac{{L_2} + 2 \bar{\sigma}_1 + \tau_0 }{6}D^{3} + \, L_1 D^2 + 8 \kappa_\theta D^3 \right)$.
\end{theorem}

\subsection{High-Order Adaptive Accelerating Method}\label{Sec: High-order-method}
To utilize high-order information, we let
$$
\overline{m}(\y; \x) = \tilde{f}_p(\y;\x).
$$
Then by involking \eqref{Inequality:Objective-Lipschitz-Continuous} and \eqref{Inequality:Gradient-Lipschitz-Continuous},  (i) and (ii) in Definition  \ref{Def:Effective-Objective-All} are satisfied with
$$
\kappa_p = \beta_p = \frac{L_p}{(p+1)!}, \quad \beta_p = \frac{L_p}{p!}, \quad \bar \kappa_p=0.
$$
Unfortunately, $\overline{m}(\y; \x)$ is not necessarily convex in this case. {According to Theorem~1 in \cite{Nesterov-2018},
	$$
	\overline{m}(\y; \x)+ \frac{\sigma \left\|\y - \x\right\|^{p+1}}{p+1}
	$$
	is a convex function when $\sigma \ge \frac{L_p}{(p-1)!}$.
	However, the choice of $\sigma$ is dependent on the problem parameter $L_p$. Moreover, checking the convexity of a polynomial function is NP hard in general \cite{ahmadi2013np}, and it remains a challenging task even when the polynomial is well structured (for instance, a sum of squares \cite{jiang2017cones}).
	Fortunately, as shown in the proof of Theorem \ref{Thm:base-case}, only the convexity of
	$$
	m(\y; {\x_0}, \sigma) = \overline{m}(\y; {\x_0})+ \frac{\sigma \left\|\y - {\x_0}\right\|^{p+1}}{p+1} + r(\y)
	$$
	at point $\bar {\x}_0$ suffices to get an upper bound of $\psi_0(\z, \tau_0)$, where $\bar {\x}_0$ is the output 
	solution of \textsf{SAS}.
	Note that
	\begin{eqnarray}
	\nabla^{2}\left(\frac{\sigma \left\|\y - \x\right\|^{p+1}}{p+1}\right) &=& \sigma(p-1)\|\y - \x\|^{p-3}(\y - \x)(\y - \x)^{\top} + \sigma \|\y - \x\|^{p-1} I \nonumber \\
	& \succeq& \sigma \|\y - \x\|^{p-1} I. \label{Hessian-p-power}
	\end{eqnarray}
	Therefore,
	$m(\y; {\x_0}, \sigma)$ is convex at $\bar{\x}_0$ as long as $\sigma \ge -\lambda_{\min}\left(\nabla^2\overline{m}(\bar{\x}_0;{{\x_0}}) \right)/\|\bar{\x}_0 - {\x_0}\|^{p-1}$. {Recall that} $\sigma^{SAS}$ is the adaptive regularizing parameter associated with $(\bar{\x}_0;{\x})$. Then we can let the input adaptive parameter of \textsf{AAS} be
	\begin{equation}\label{sigma-0-AAS}
	\sigma_0^{AAS}=\max\{\sigma^{SAS},  -\lambda_{\min}\left(\nabla^2\overline{m}(\bar{\x}_0;{{\x_0}}) \right)/\|\bar{\x}_0 - {\x_0}\|^{p-1}\}
	\end{equation}
	to guarantee the convexity of $m(\y; {\x}_0, \sigma_0^{AAS})$ at $\bar{\x}_0$. Moreover, we have that $F(\bar{\x}_0) < m(\bar{\x}_0; {\x}_0, \sigma^{SAS}) \le  m(\bar{\x}_0; {\x}_0, \sigma_0^{AAS})$, and so $\bar{\x}_0$ is still a successful iterate in \textsf{SAS}.
	In practice, we may further add a small positive number to $\sigma_0$ to rid 
	ill-conditioning caused by the numerical error when computing $-\lambda_{\min}\left(\nabla^2\overline{m}(\bar{\x}_0;{\x}) \right)/\|\bar{\x}_0 - \x\|^{p-1}$. Finally, we
	arrive at the following iteration bound.
}

\begin{theorem}\label{Thm:d-order} Letting $\overline{m}(\y; \x) = \tilde{f}_p(\y;\x)$ in \textsf{UAA}, we obtain an adaptive accelerating $p$-th order method, and  the total iteration number of getting an $\epsilon$-optimal solution is
	\begin{eqnarray*}
		& & 2 + \frac{2}{\log(\gamma_1)}\log\left(\frac{\bar{\sigma}_1}{\sigma_{\min}}\right) + \left\lceil \frac{1}{\log\left(\gamma_3\right)} \log\left(\frac{2\left(\frac{L_p}{(p+1)!} + \bar{\sigma}_2 + \kappa_{\theta}\right)^{p+1} p^{p-1}}{\eta^p (p-1)!\tau_0}\right) \right\rceil  \\
		& & + \, \left(1+\frac{2}{\log(\gamma_1)}\log\left(\frac{\bar{\sigma}_2}{\sigma_{\min}}\right)\right)\left[1 + \left(\frac{C_p}{\varepsilon}\right)^{\frac{1}{p+1}} \right]
	\end{eqnarray*}
	where
	\[ C_p = (p+1)! \left( \frac{\frac{2 L_p}{p!} + 2{\hat{\sigma}_1} + {\hat \sigma_2}}{2(p+1)}D^{p+1}  +  (\kappa_\theta + {\hat{\sigma}_1} )  (2D)^{p+1} \right) . \]
	
\end{theorem}

\section{Numerical Experiments}\label{section:experiment}
In this section, we present the results of some numerical experiments for solving the $\ell_1$/$\ell_2$-regularized logistic regression problems. { The reason for this choice is that the logistic loss function is known to be convex, and the corresponding $\ell_1$/$\ell_2$-regularized problems are convex as well. Thus, the proposed methods are directly applicable. Moreover, these two problems are common in testing the performance of various second-order methods in the literature; see \cite{Ghanbari-Scheinberg-2016-Practical,Scheinberg-2016-Practical} for $\ell_1$-regularized problem and \cite{Berahas-Bollapragada-Nocedal-2020, Bollapragada-Byrd-Nocedal-2019} for $\ell_2$-regularized problem. All the experiments are conducted on a MacBook Pro with Mac OS High Sierra 10.13.6, a Intel i5 2.6GHz CPU  and 16GB memory.}
\subsection{$\ell_2$-Regularized Logistic Regression Problem}
We first test the performance of the algorithms by evaluating the following $\ell_2$-regularized logistic regression problem
\begin{equation}\label{Prob:GGLR}
\min_{\x\in \br^d} \ f(\x) = \frac{1}{n}\sum\limits_{i=1}^n \ln \left( 1+\exp\left(-b_i \cdot {\sa}_i^\top \x\right) \right) + \frac{\lambda}{2}\| \x\|^2
\end{equation}
where $(\sa_i,b_i)_{i=1}^n$ is the samples in the data set, and the regularization parameter is set as $\lambda=10^{-5}$. To observe the acceleration,  the starting point is randomly generated from a Gaussian random variable with zero mean and a large variance (say $5000$). In this way, initial solutions are likely to be far away from the global solution.
\begin{table}[!t]
	\caption{Statistics of datasets for $\ell_2$-regularized logistic regression.}
	\begin{center}
		\begin{tabular}{|c|c|c|} \hline
			Dataset & $n$ & $d$ \\ \hline
			\textsf{a9a} & 32,561 & 123 \\
			\textsf{phishing} & 11,055 & 68 \\
			\textsf{sonar} & 208 & 60 \\
			\textsf{svmguide3} & 1,243 & 22 \\
			\textsf{w8a} & 49,749 & 300 \\
			\textsf{SUSY} & 5,000,000 & 18 \\ \hline
		\end{tabular}\vspace{-1em}
	\end{center}
	\label{Table: dataset}
\end{table}

We implement a variant of Algorithm \ref{Algorithm:Framework} with cubic regularization, referred to as {\it Adaptively Accelerated Cubic Regularized}\/ (AARC) Newton's method. In this variant we {set $\sigma_0 = \tau_0 = 1$, $\sigma_{\min} = 10^{-16}$, $\kappa_\theta=0.1$, $\gamma_1=\gamma_2=\gamma_3=2$ and $\eta = 0.01$. We first run Algorithm \ref{Algorithm:Framework} and switch to the adaptive cubic regularization phase of Newton's method (ARC) in \cite{Cartis-2011-Adaptive-I, Cartis-2011-Adaptive-II} when the iterates are getting close to the global optimum. In particular, the switch is activated after $10$ successful iterations of {\it Accelerated Adaptive Subroutine}\/ are performed and the progress made by each iteration is small, i.e., $\frac{\left| f(\x_{k+1})-f(\x_k) \right|}{\left| f(\x_k)\right|} \leq 0.1$. The final stopping criterion is set to be $\left\|\nabla f(\x)\right\|\leq 10^{-9}$ after switching to the ARC phase.} In the implementation, we apply the so-called Lanczos process to approximately solve the subproblem $\min_{\y \in \br^d} m(\y;\x_i,\sigma_i)$. In addition to \eqref{Criterion:Approximate-Adaptive}, the approximate solution $\s$ is also made to satisfy
\begin{equation}\label{AM:1}
(\y - \x_i)^\top \nabla f(\x_i) + (\y - \x_i)^\top \nabla^2 f(\x_i) (\y - \x_i) + \sigma\left\|\y - \x_i\right\|^3 = 0
\end{equation}
for given $\x_i$ and $\sigma_i$. Note that \eqref{AM:1} is a consequence of the first order necessary condition, and as shown in Lemma 3.2 \cite{Cartis-2011-Adaptive-I}, the global minimizer of $m(\y;\x_i, \sigma_i)$ when restricted to a Krylov subspace $$\KCal := \text{span}\{\nabla f(\x_i), \nabla^2 f(\x_i) \nabla f(\x_i), \left( \nabla^2f(\x_i) \right)^2 \nabla f(\x_i), \ldots\}$$ satisfies \eqref{AM:1} independent of the subspace dimension.
Minimizing $m(\y;\x_i,\sigma_i)$ in the Krylov subspace is also computationally favorable, as it can be done at the cost of $O(d)$ involving only factorizing a tri-diagonal matrix. Thus, the associated approximate solution can be found through the so-called Lanczos process, where the dimension of $\KCal$ is gradually increased and an orthogonal basis of each subspace $\KCal$ is built up which typically involves one matrix-vector product.  Condition \eqref{Criterion:Approximate-Adaptive} can be used as the termination criterion for the Lanczos process in the hope to find a suitable trial step before the dimension of $\KCal$ approaches $d$.

We compare the new AARC method with {4 other methods including: the adaptive cubic regularized Newton's method (ARC), the trust region method (TR), the limited memory Broyden-Fletcher-Goldfarb-Shanno method (L-BFGS), and the adaptive gradient method (AGD). We adopt the implementation of ARC and TR in the public package\footnote{https://github.com/dalab/subsampled{\_}cubic{\_}regularization} with the default parameters except the full rather than the subsampled batch of the component functions is taken in ARC and the upper bound on the radius of trust region in TR is set to be $10^{4}$. To implement the L-BFGS method, we use the Wolfe conditions to perform the line search and set the descent parameters in Armijo rule,  the curvature condition and the memory size as $0.01$, $0.9$ and $50$ respectively. AGD is implemented based on AdaGrad in \cite{Duchi-2011-Adaptive}.}
The experiments are conducted on 6 \textsf{LIBSVM Sets} \footnote{https://www.csie.ntu.edu.tw/\~{}cjlin/libsvm/} for binary classification, and the summary of those datasets are shown in Table \ref{Table: dataset}.

The results in Figure \ref{fig1} and Figure \ref{fig2} confirm that AARC indeed accelerates ARC, especially when the current iterate has not entered the local region of quadratic convergence yet. Furthermore, AARC outperforms other methods in both computational time and iterations counts in {the given datasets}. Compared to TR with a local constrained quadratic model, AARC achieves more progress and cheaper per-iteration cost at each iteration because of the advantage of a unconstrained local cubic approximation model~\eqref{Definition:Approximate-Model} and the flexible stopping criterion~\eqref{Criterion:Approximate-Adaptive}; compared to L-BFGS, AARC suffers from relatively higher per-iteration cost but its solution can achieve higher accuracy.
\begin{figure}[!t]
	\includegraphics[width=0.48\textwidth, height=4.8cm]{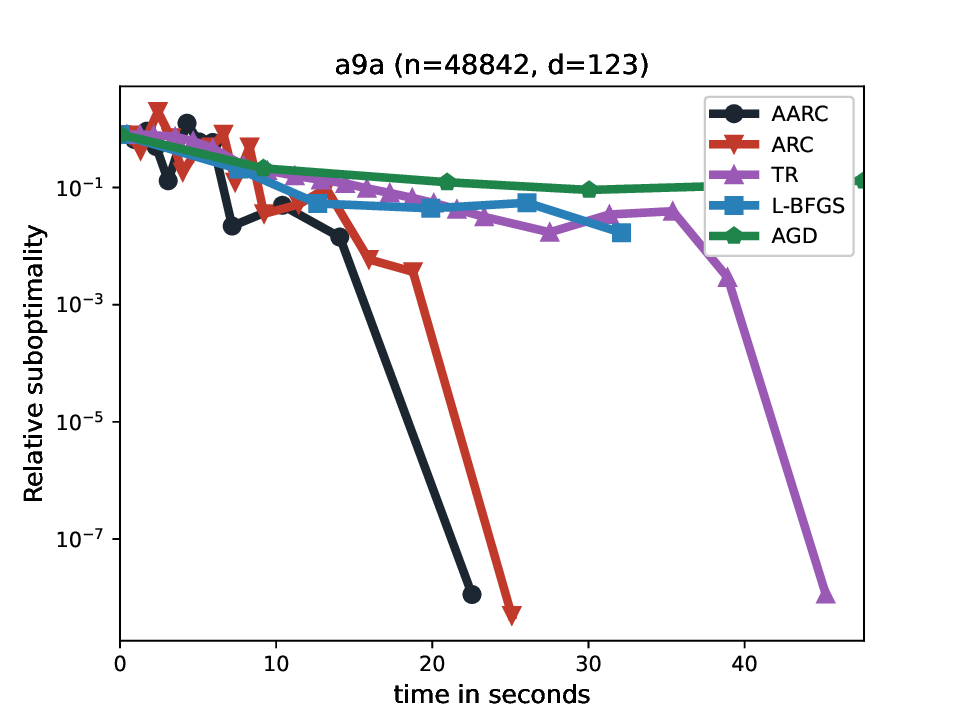}
	\includegraphics[width=0.48\textwidth, height=4.8cm]{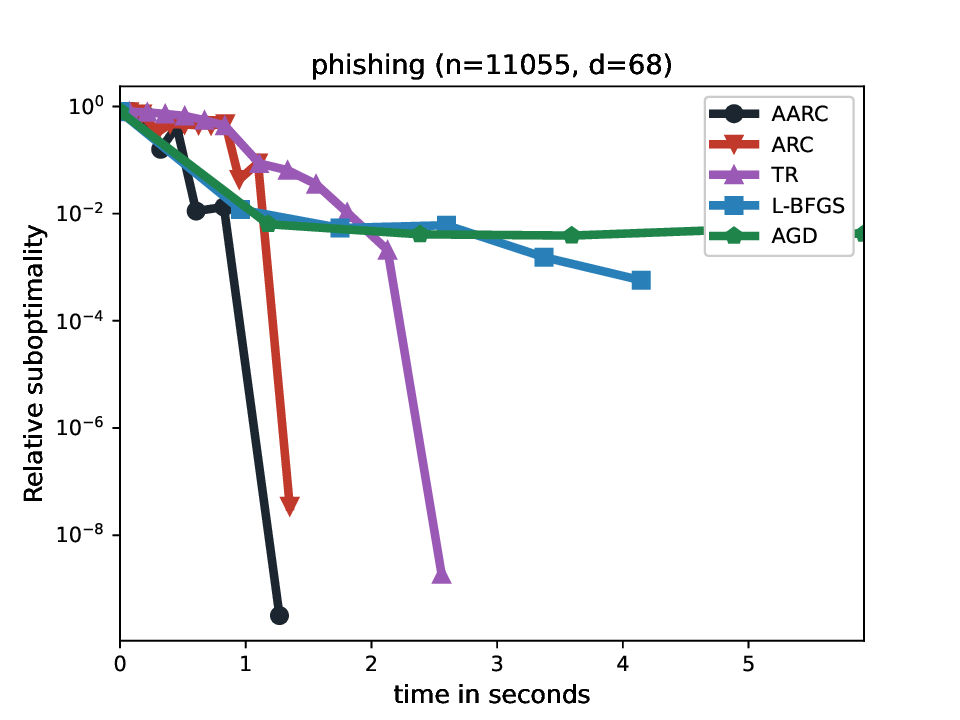}
	\includegraphics[width=0.48\textwidth, height=4.8cm]{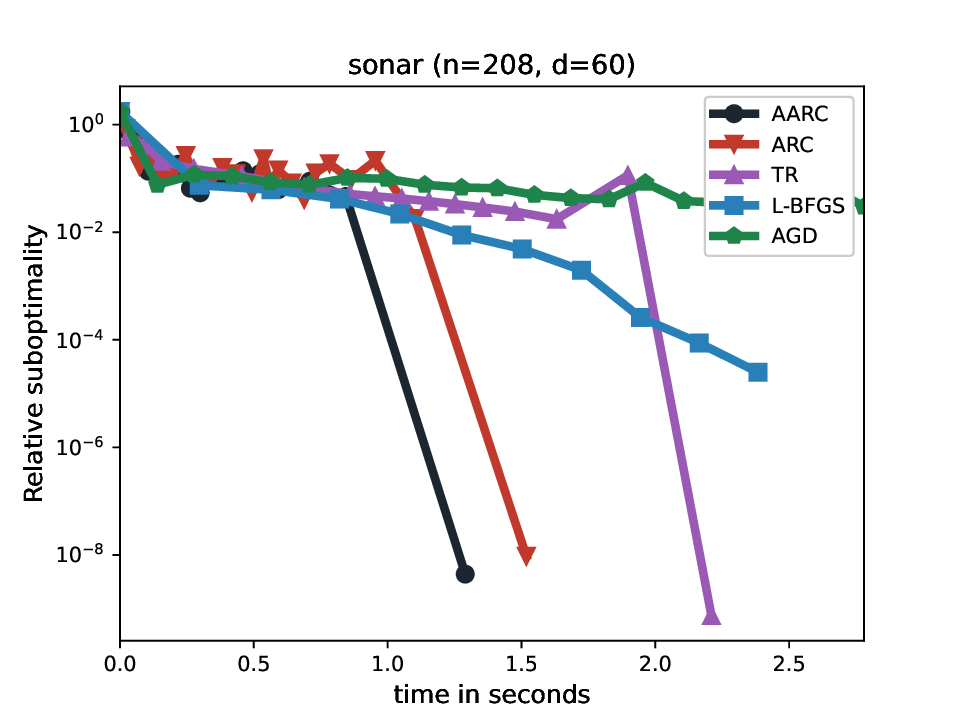}
	\includegraphics[width=0.48\textwidth, height=4.8cm]{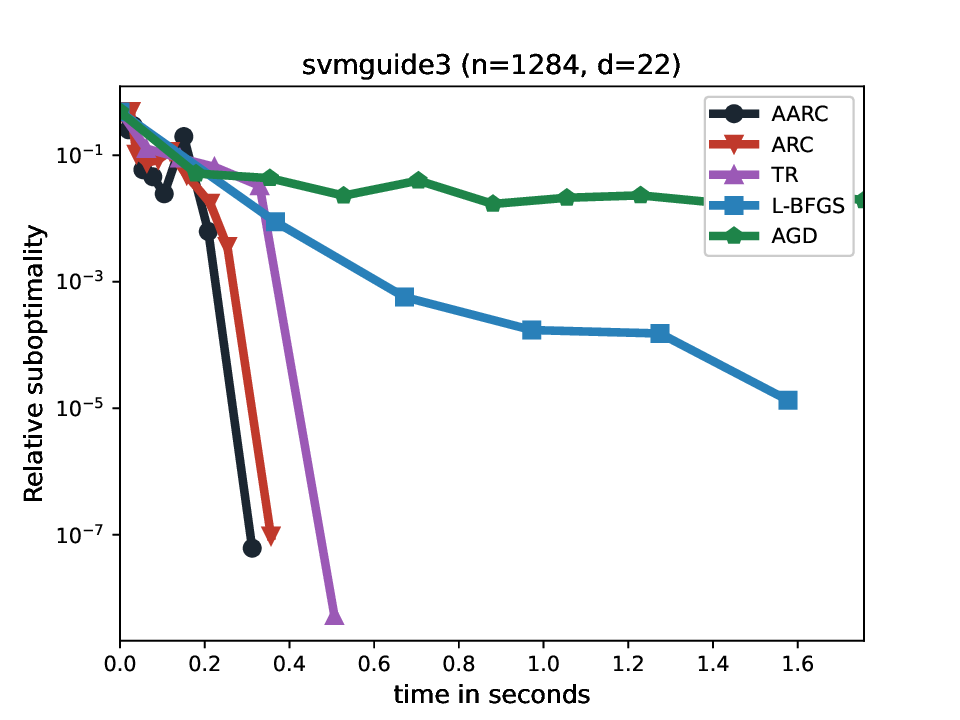}
	\includegraphics[width=0.48\textwidth, height=4.8cm]{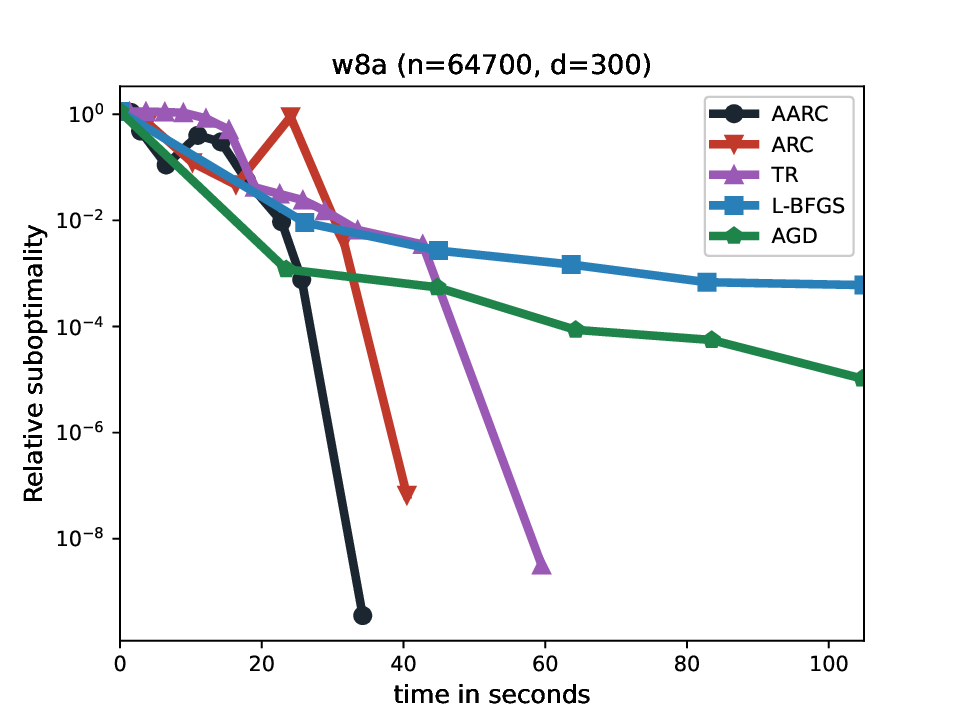} \quad
	\includegraphics[width=0.48\textwidth, height=4.8cm]{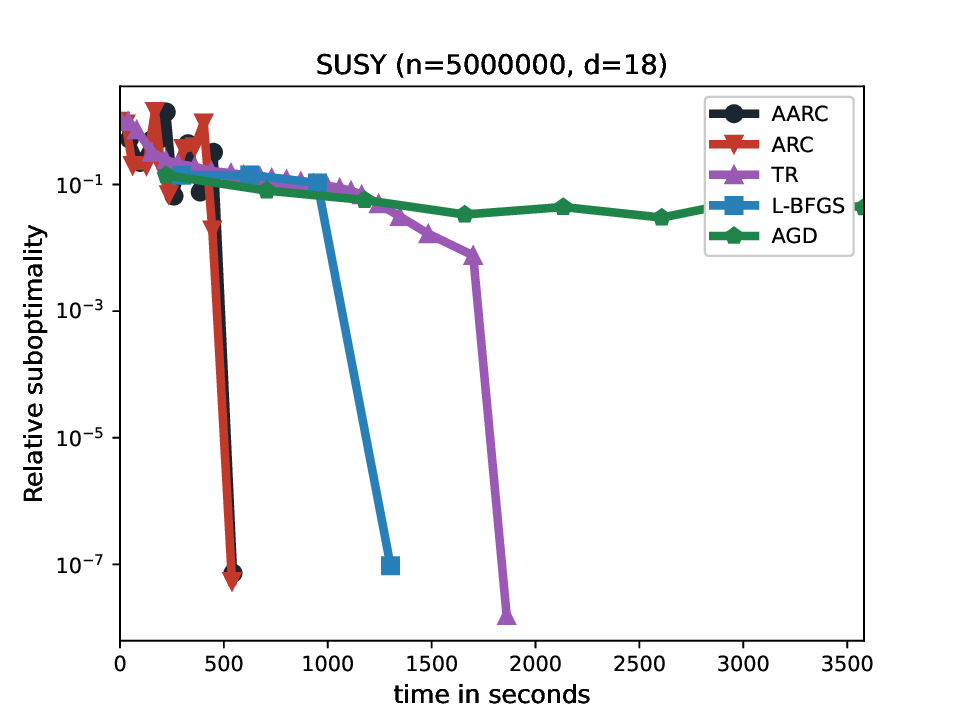}
	\caption{Performance of AARC and all benchmark methods on the task of $\ell_2$-regularized logistic regression  (loss vs.\ time)}
	\label{fig1}
\end{figure}
\begin{figure}[!t]
	\includegraphics[width=0.48\textwidth, height=4.8cm]{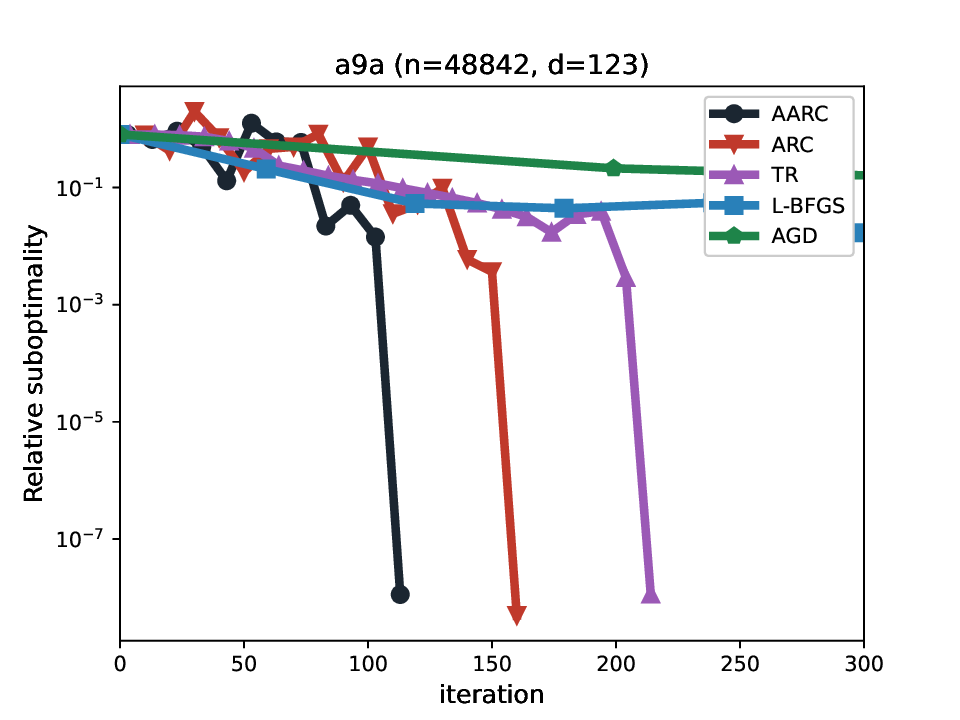}
	\includegraphics[width=0.48\textwidth, height=4.8cm]{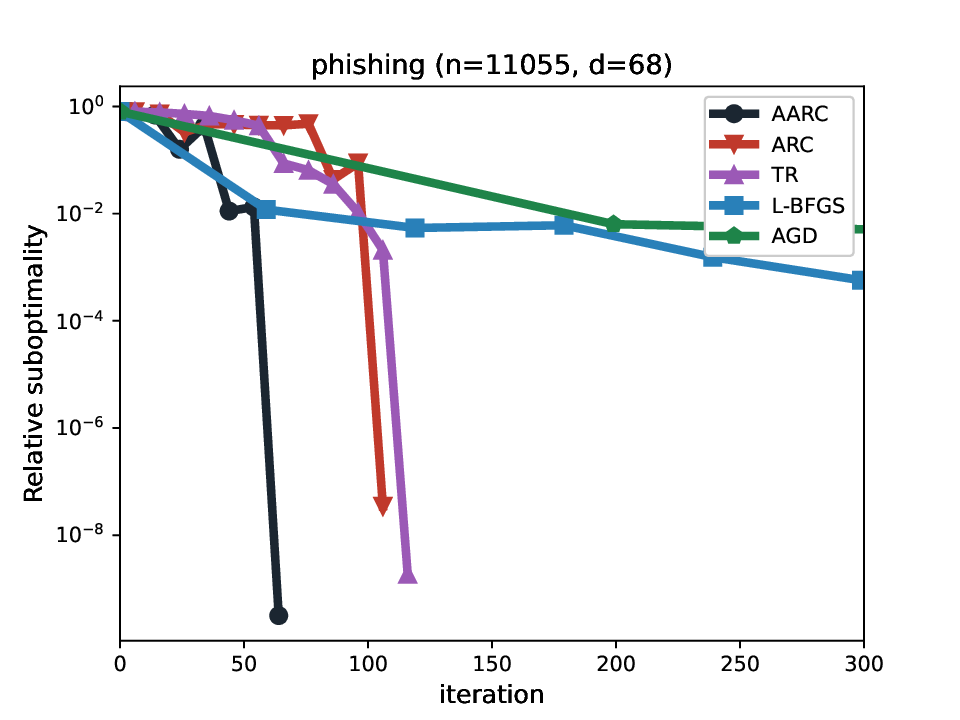}
	\includegraphics[width=0.48\textwidth, height=4.8cm]{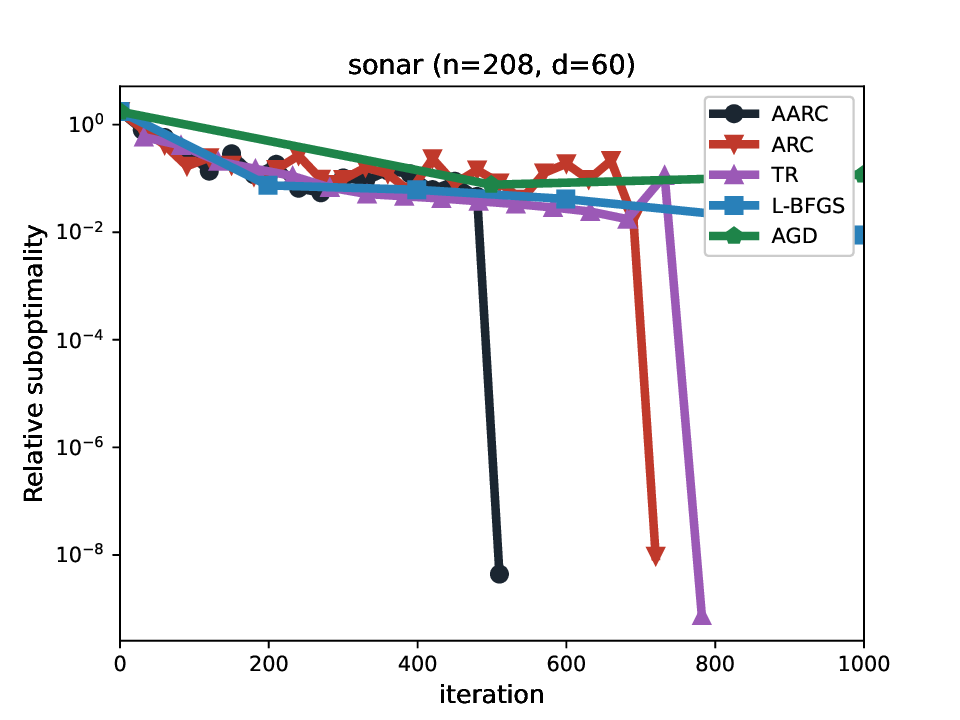}
	\includegraphics[width=0.48\textwidth, height=4.8cm]{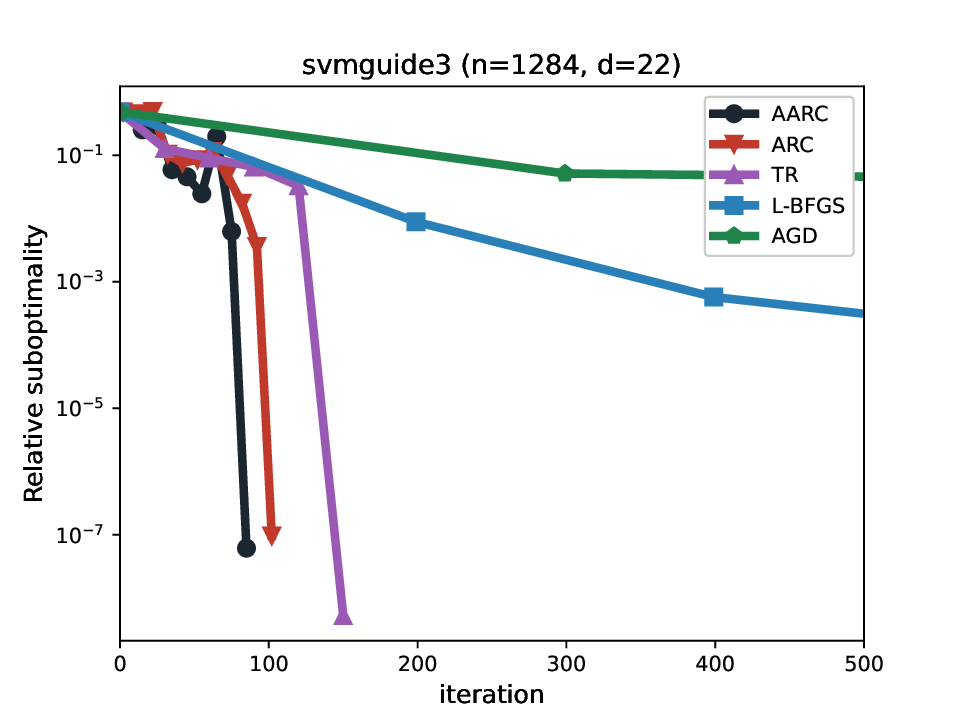}
	\includegraphics[width=0.48\textwidth, height=4.8cm]{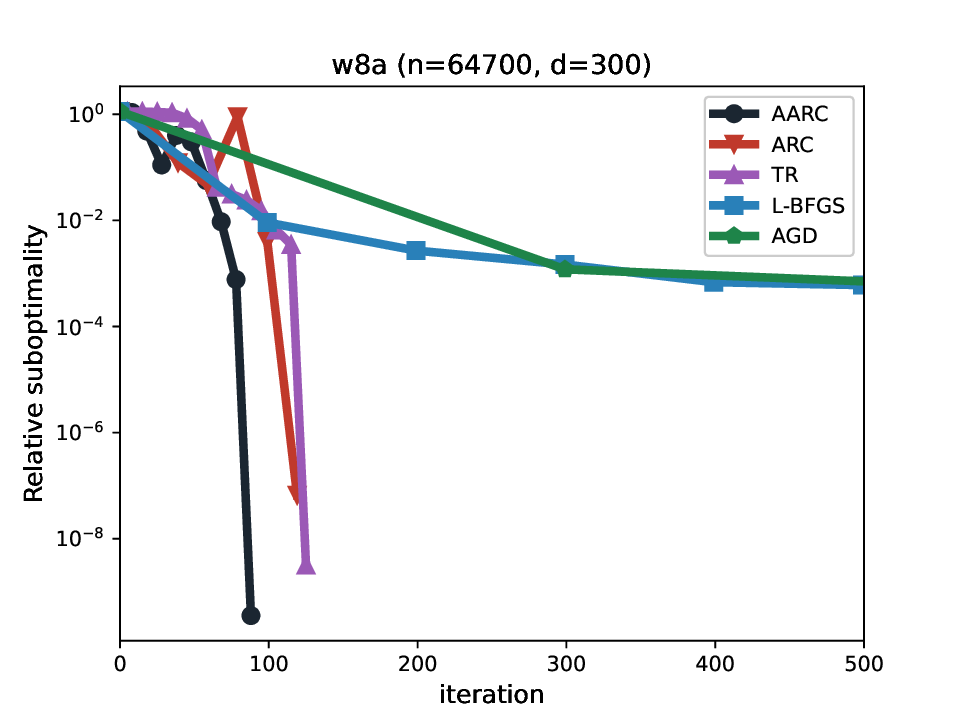} \quad
	\includegraphics[width=0.48\textwidth, height=4.8cm]{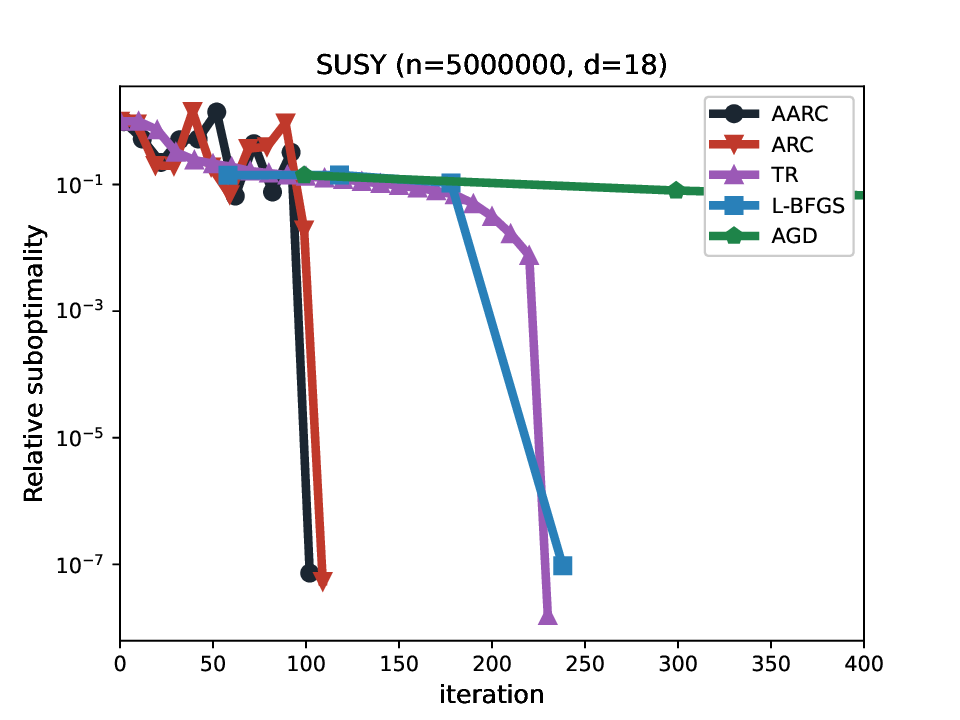}
	\caption{Performance of AARC and all benchmark methods on the task of $\ell_2$-regularized logistic regression (loss vs.\ iterations)}
	\label{fig2}
\end{figure}

\subsection{$\ell_1$-Regularized Logistic Regression Problem}\label{Subsec:l1-numerical}
Then we test the algorithms on the following $\ell_1$-regularized logistic regression problem:
\begin{equation}\label{prob:logistic-regression}
\min_{\x\in\br^d} \ \frac{1}{n}\sum_{i=1}^n \log\left(1+\exp(-y_i \w_i^\top\x)\right) + \lambda\left\|\x\right\|_1,
\end{equation}
where $\{\left(\w_i, y_i\right)\}_{i=1}^n$ is
a collection of data samples with $y_i \in \{-1, 1\}$ being the label. The regularization term $\left\|\x\right\|_1$ promotes sparse solutions and $\lambda>0$ balances sparsity with goodness-of-fit and generalization. In addition, $\lambda$ was chosen by \textsf{LIBLINEAR} with five-fold cross validation. The experiments are conducted on 3 data sets that all come from \textsf{LIBSVM}\footnote[1]{The collection is available at https://www.csie.ntu.edu.tw/$\sim$cjlin/libsvmtools/datasets}, and the summary of those datasets are shown in Table \ref{Table:dataset2}.

We first test how the inexactness of the Hessian matrix affects the performance of adaptive accelerated proximal cubic regularization of Newton method (AARC) on $\ell_1$-regularized logistic regression problem~\eqref{prob:logistic-regression}. In particular, we implement inexact AARC with different values of $\kappa_{hs}$ in \eqref{kappa-hs} and set the step size to construct the approximated Hessian as $h_i = \min\{ \kappa_{hs}, \kappa_{hs}\left\|\x_{i} - \x_{i-1}\right\| \}$, where $\left\|\x_{i} - \x_{i-1}\right\|$ is the size of difference of the last two consecutive iterates. We plot relative suboptimality versus iteration counts on all datasets in Figure \ref{fig4}, but we do not plot the figures regarding the run-time, as all the methods in Figure \ref{fig4} solve similar subproblems and the run-time is proportional to the iteration counts. Figure \ref{fig4} indicates that inexact AARC works well in general, and the corresponding iteration complexity decreases as the value of $\kappa_{hs}$ decreases, which makes sense and implies that
more accuracy of the Hessian leads to faster convergence of the proposed algorithm. Note that we do not choose a very small value of $\kappa_{hs}$ because if we do then the corresponding curves will be very close to that of exact AARC, making it hard to distinguish the two curves.

\begin{table}[!t]
	\begin{center}\vspace*{-.5em}
		\begin{tabular}{|r|c||ccc|c|} \hline
			Name & Description & $n$ & $d$ & Scaled Interval & $\lambda$ \\ \hline
			\textsf{a9a} & UCI adult & 48842 & 123 & $\left[0, 1\right]$ & 4.5e-03 \\
			\textsf{covetype} & forest covetype & 581012 & 54 & $\left[0, 1\right]$ & 2.6e-03 \\
			\textsf{w8a} & - & 64700 & 300 & $[0,1]$ & 7.0e-04 \\ \hline
		\end{tabular}
	\end{center}
	\caption{Statistics of datasets for $\ell_1$-regularized logistic regression.}\vspace*{-2em}\label{Table:dataset2}
\end{table}

We compare the AARC with Nesterov's accelerated gradient method (Nesterov83) (adapted for composite optimization), fast iterative shrinkage thresholding algorithm (FISTA)~\cite{Beck-2009-Fast}, and the accelerated regularized Newton methods proposed by Grapiglia and Nesterov (GN)~\cite{Grapiglia-Nesterov-2018} on $\ell_1$-regularized logistic regression problem~\eqref{prob:logistic-regression}. We use the TFOCS\footnote{http://cvxr.com/tfocs/} implementation with default parameter settings for Nesterov83 and FISTA. Note that Nesterov83 and FISTA have different coefficients on the momentum term, and their numerical performances would behave differently as shown in Figure \ref{fig3}. For AARC, we use the same setting as that for the $\ell_2$-regularized logistic regression problem, e.g., $\sigma_0=1$, $\sigma_{\min}=10^{-16}$, $\kappa_\theta=0.1$, $\gamma_1=\gamma_2=\gamma_3=2$ and $\eta = 0.01$. The difference is that we adopt FISTA~\cite{Beck-2009-Fast} to solve the subproblem in AARC, as the subproblem itself is a convex composite optimization problem. The maximum number of iterations for solving those subproblems is 500 and the parameter setting for FISTA is default. Finally, we manage to implement the accelerated regularized Newton methods (GN) in~\cite{Grapiglia-Nesterov-2018} with two minor modifications: (i) the subproblem in GN is approximately solved with the stopting criterion \eqref{Criterion:Approximate-Adaptive}, where we set $\kappa_\theta = 10^{-20}$ such that the subproblem is almost solved exactly; (ii) the nonsmooth objective function in the auxiliary function in GN is replaced by its subgradient to avoid computing another proximal mapping by iterative algorithms for computational efficiency otherwise the per-iteration cost will be doubled.

We plot relative suboptimality versus iteration counts as well as relative suboptimality versus time on all datasets in Figure \ref{fig3}. It is clear in Figure \ref{fig3} that our method consistently outperforms Nesterov83 and FISTA in terms of the number of iterations and the overall computational time although the subproblem in AARC does not have a closed-form solution and is much more time-consuming to solve, which is in contrast with that of accelerated first order methods. { Compared to the accelerated second order method GN, AARC has slightly smaller
	iteration counts. This is possibly due to the dynamic adjustment of the adaptive parameter $\tau_{j+1}$ of the auxiliary function $\psi_{j+1}(\z, \tau_{j+1})$ in our AAS subroutine, while
	similar parameter in GN is updated by solving a certain univariate polynomial equation. Besides the slight difference in the iteration counts, GN is also more time consuming per iteration as it needs to solve the subproblem more accurately.}
\begin{figure}[!t]
	\hspace*{-2em}\includegraphics[width=0.37\textwidth, height=3.8cm]{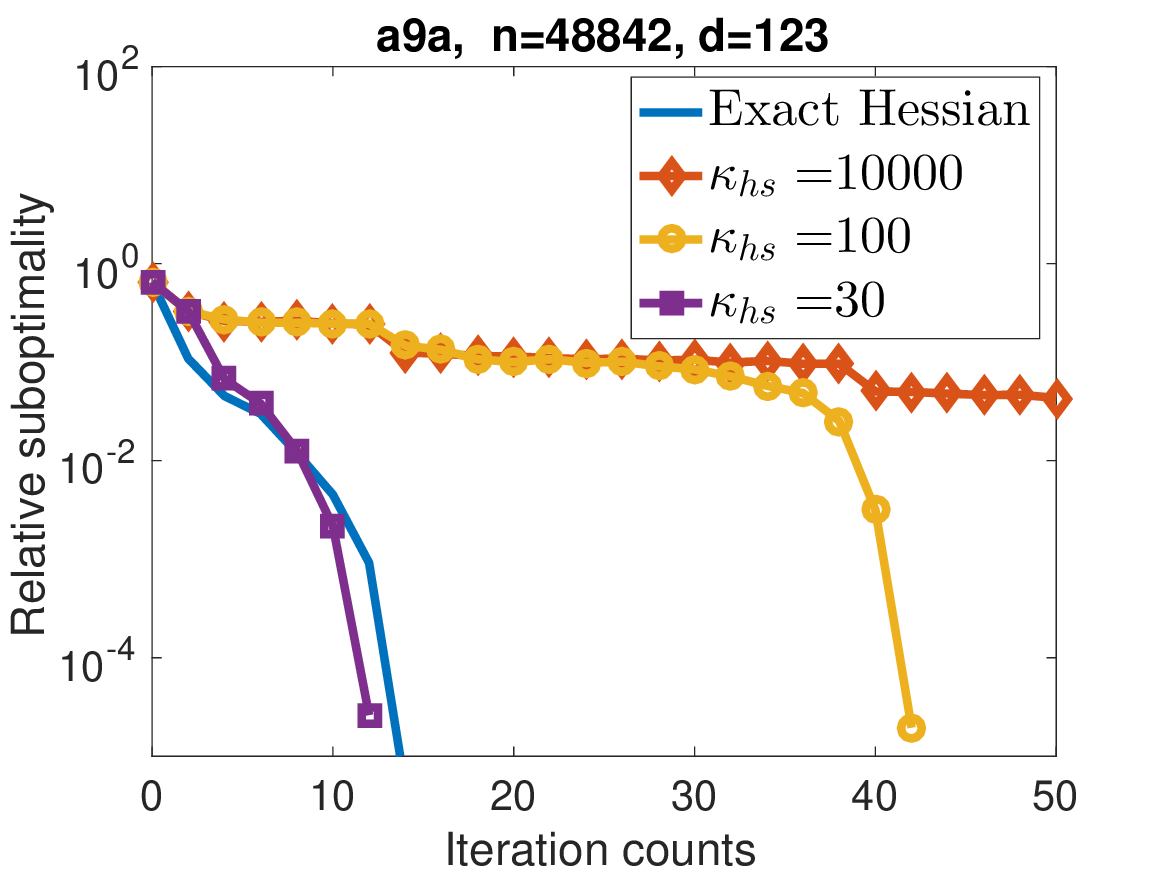}
	\hspace*{-1.5em}\includegraphics[width=0.37\textwidth, height=3.8cm]{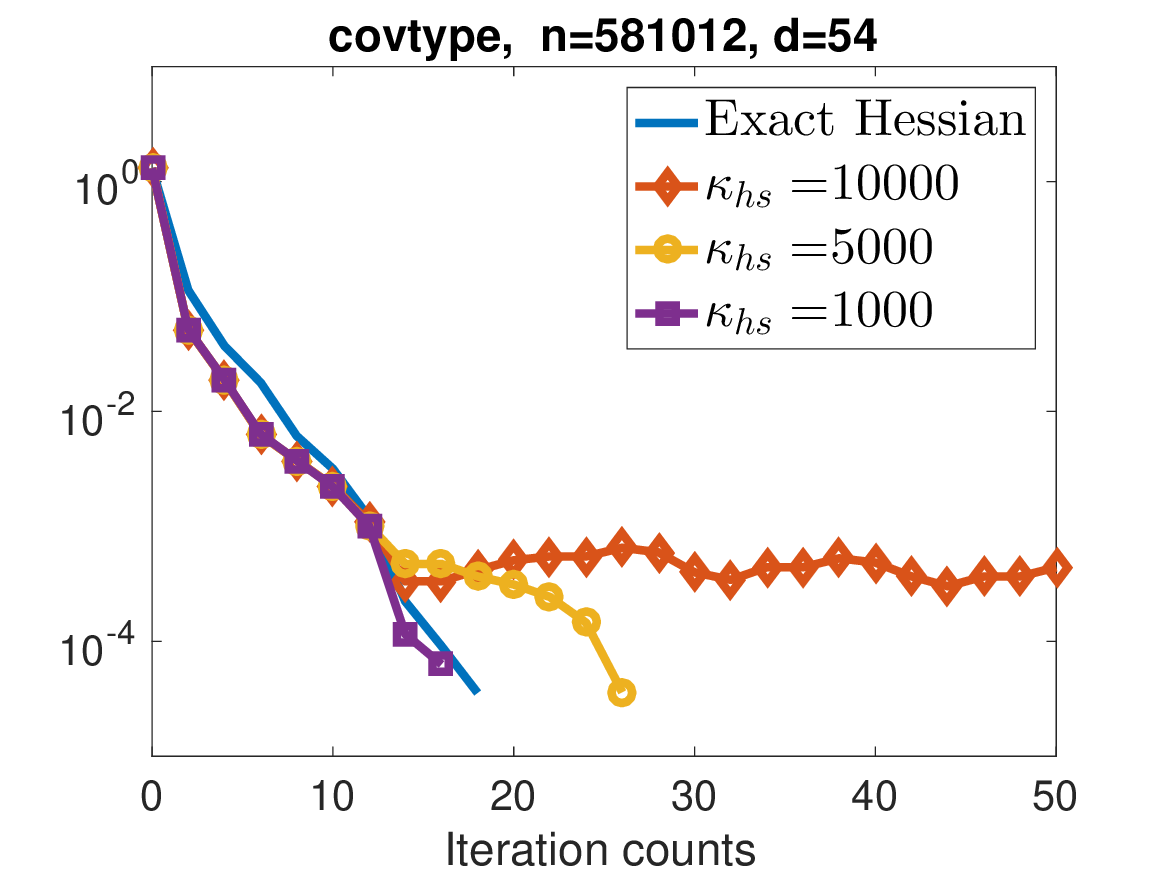}
	\hspace*{-1.5em}\includegraphics[width=0.37\textwidth, height=3.8cm]{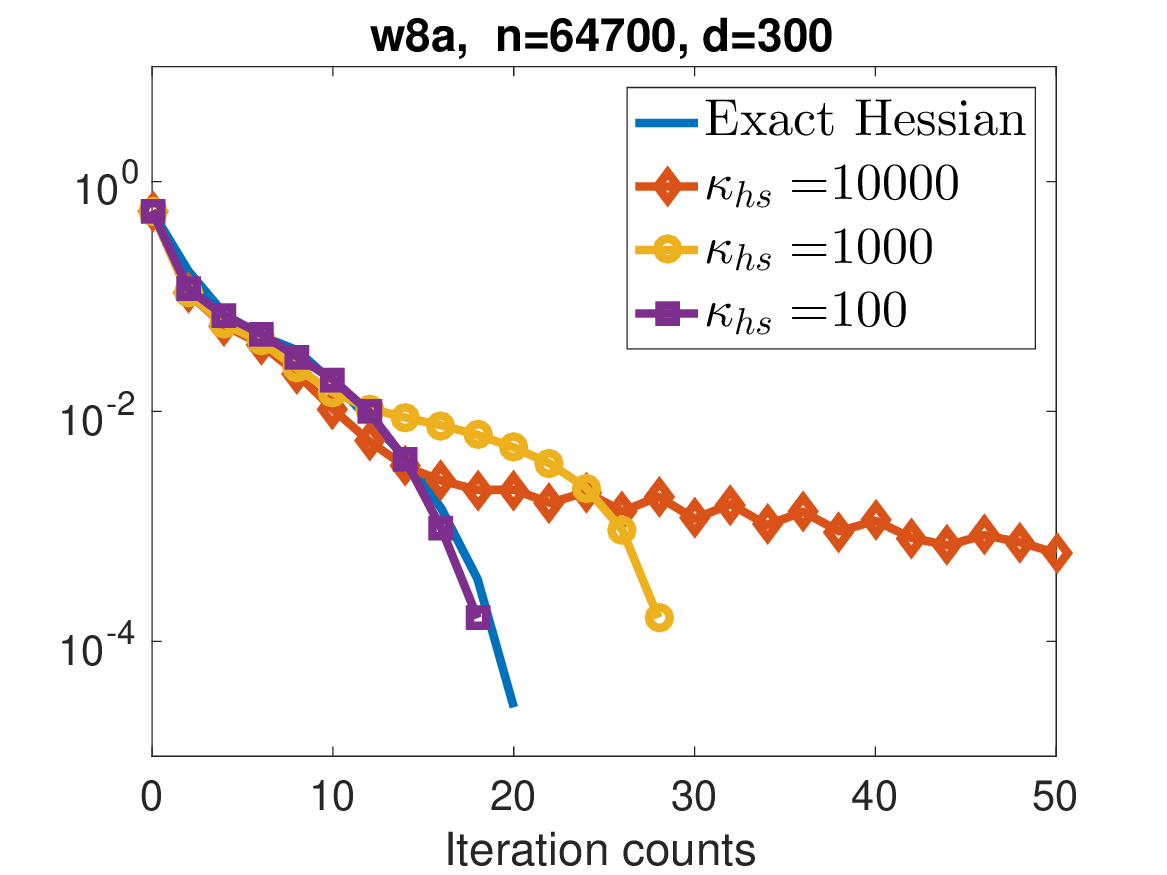}
	\caption{Iteration counts of AARC with Inexact Hessians on $\ell_1$-regularized logistic regression.}\label{fig4}
\end{figure}

\begin{figure}[!ht]
	\includegraphics[scale=0.45]{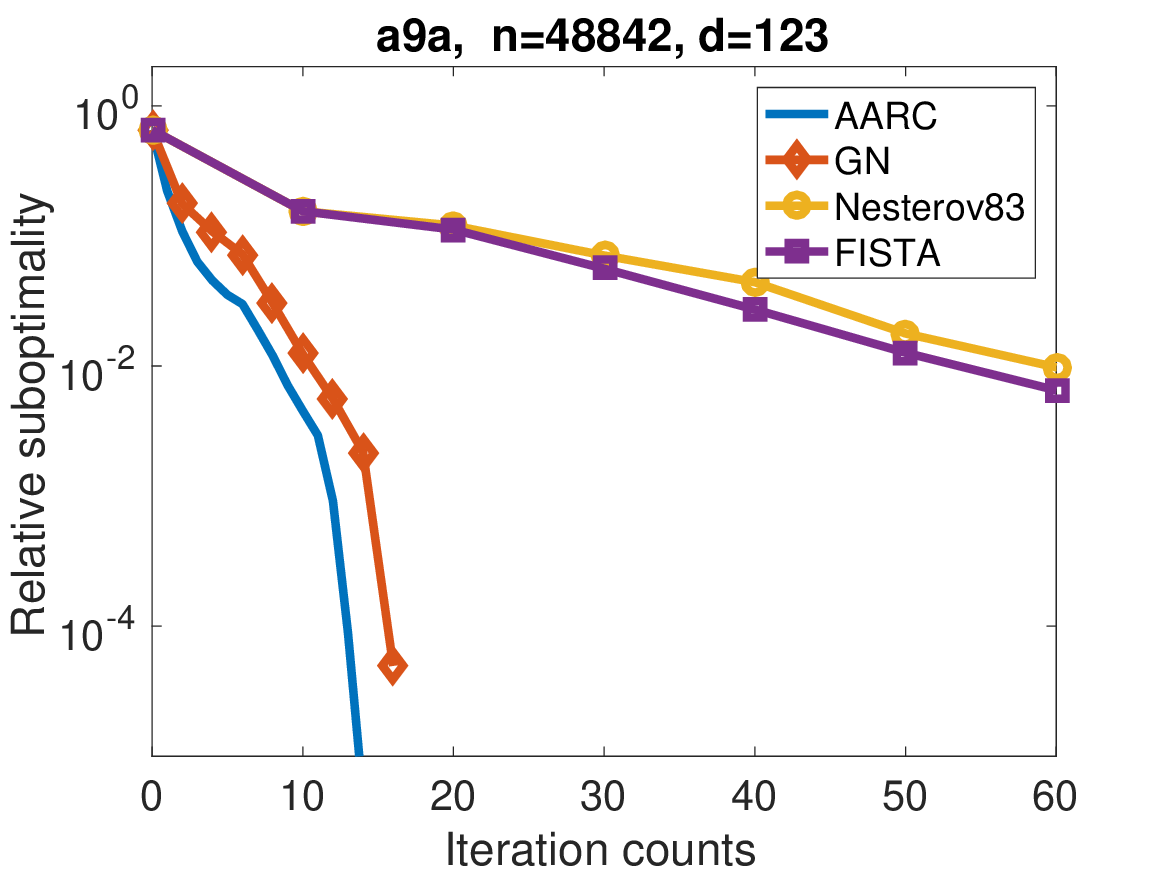}
	\includegraphics[scale=0.45]{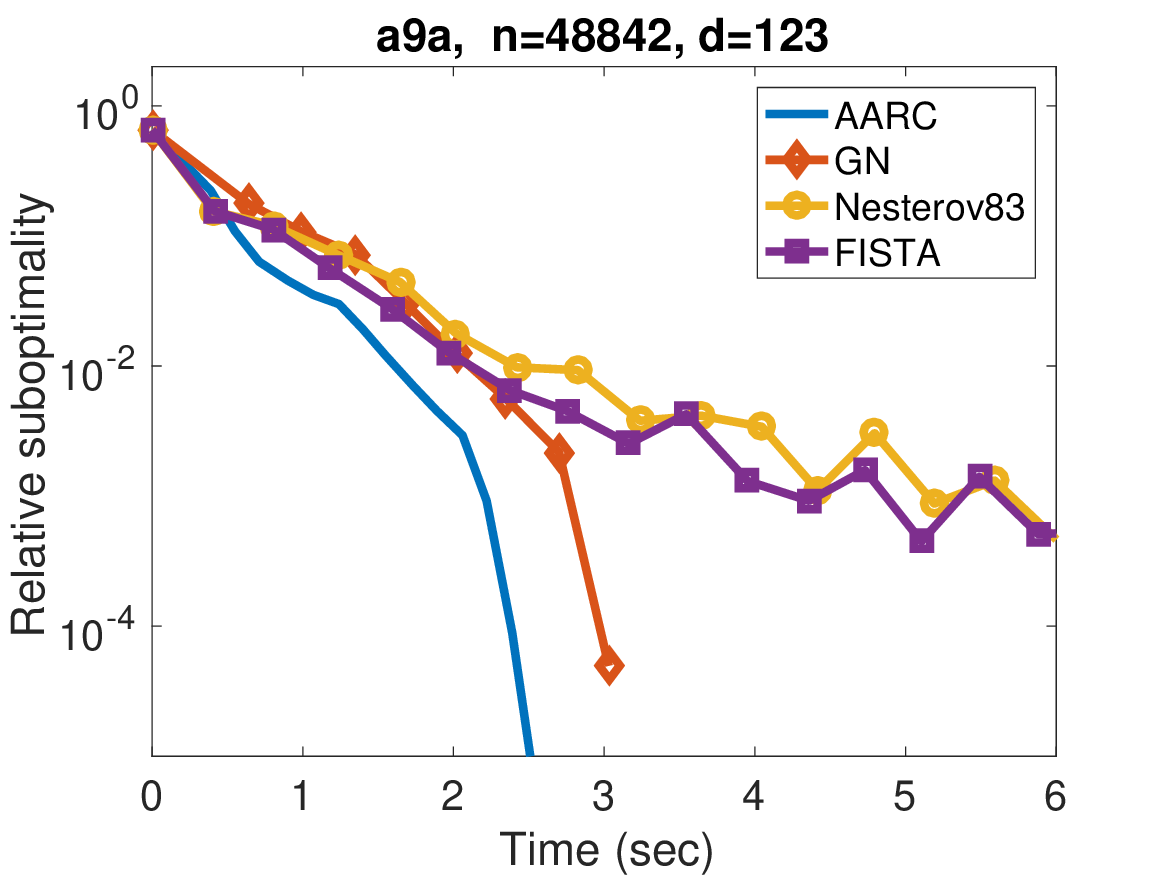}
	\includegraphics[scale=0.45]{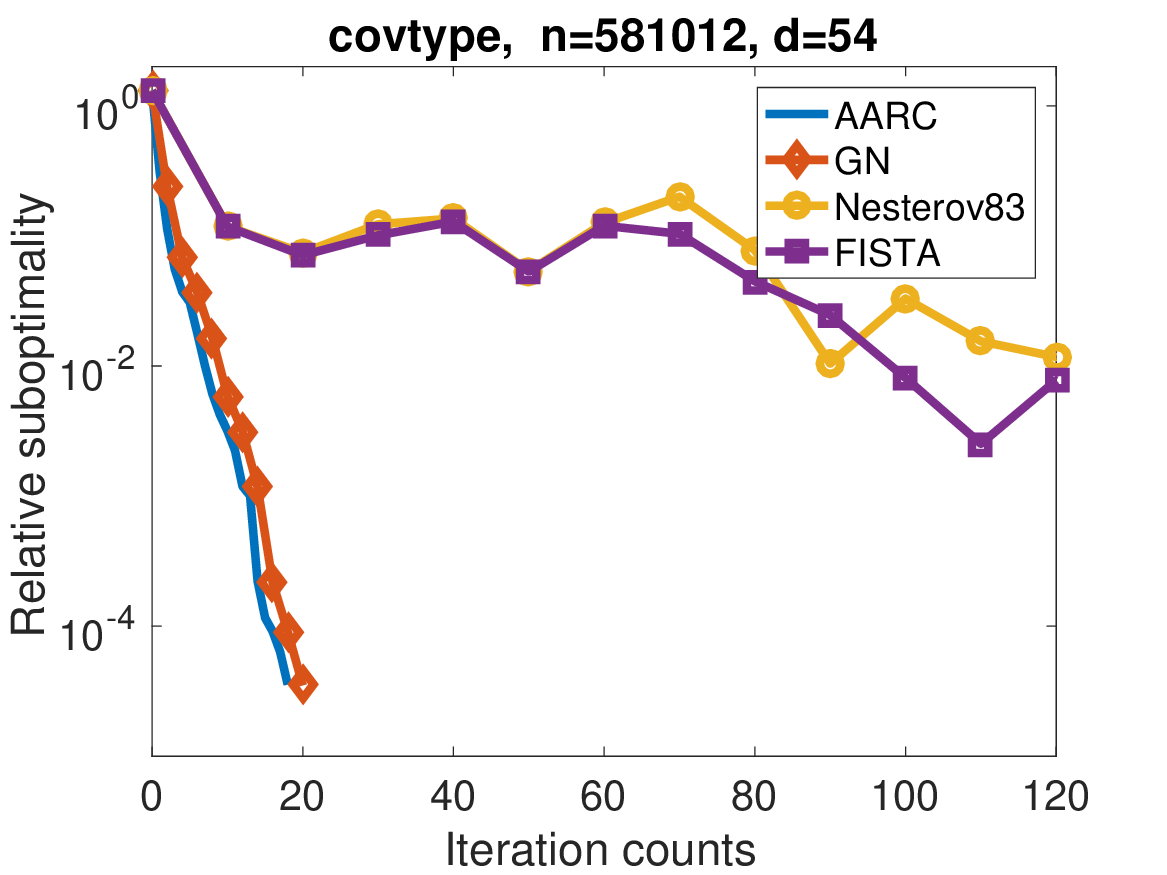}
	\includegraphics[scale=0.45]{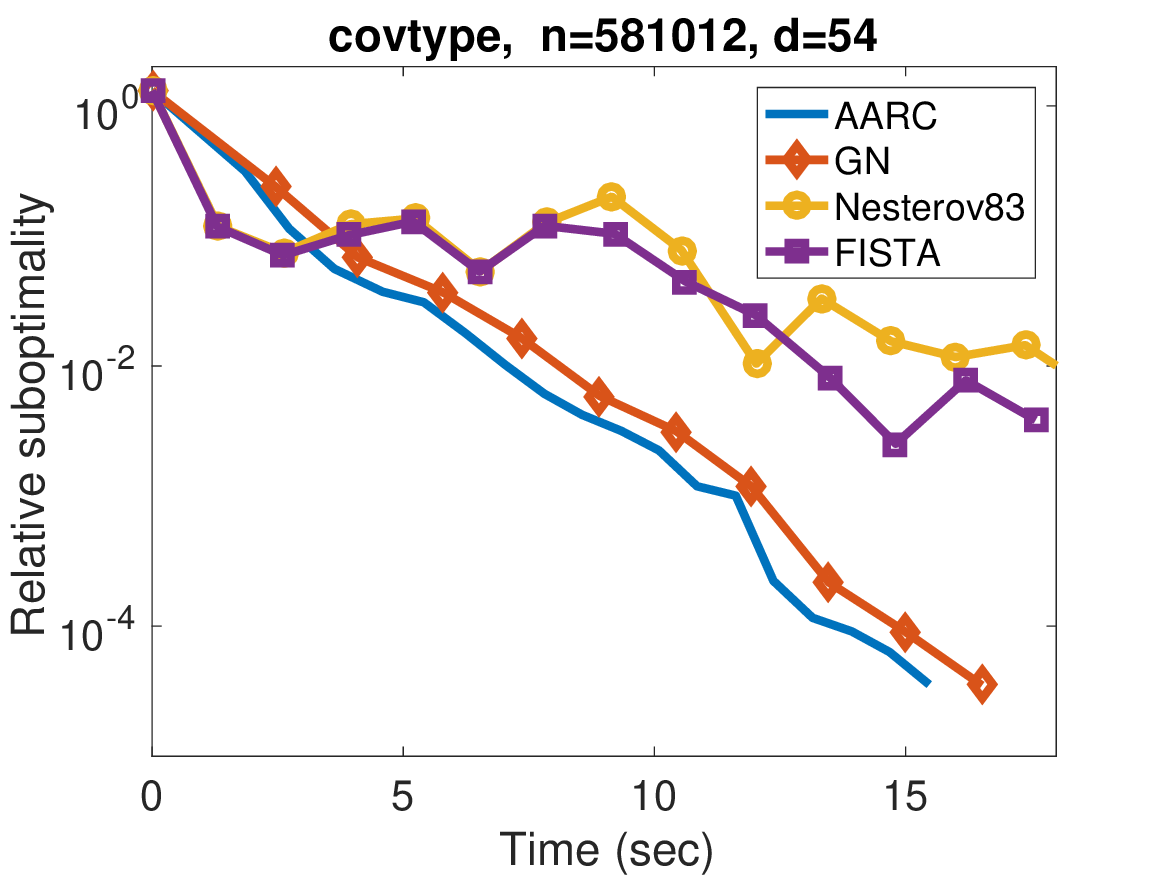}
	\includegraphics[scale=0.45]{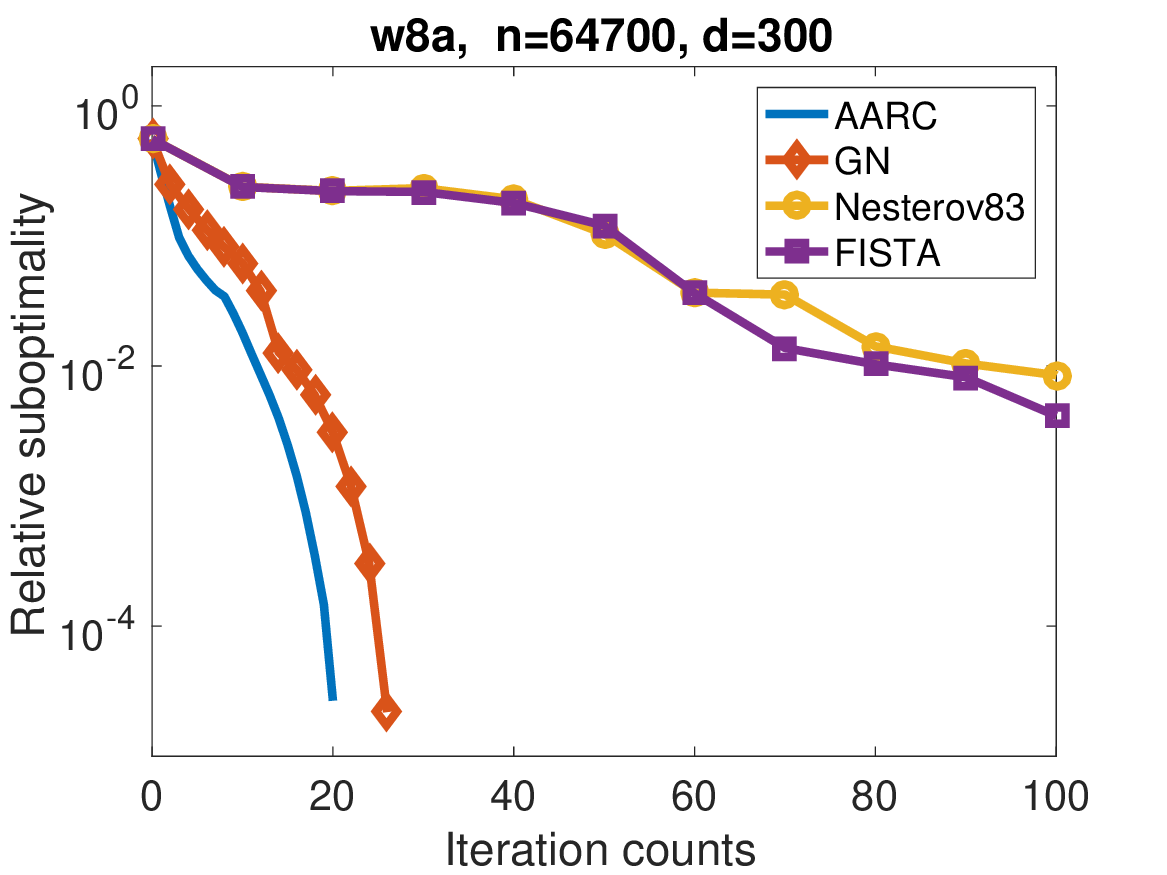}
	\includegraphics[scale=0.45]{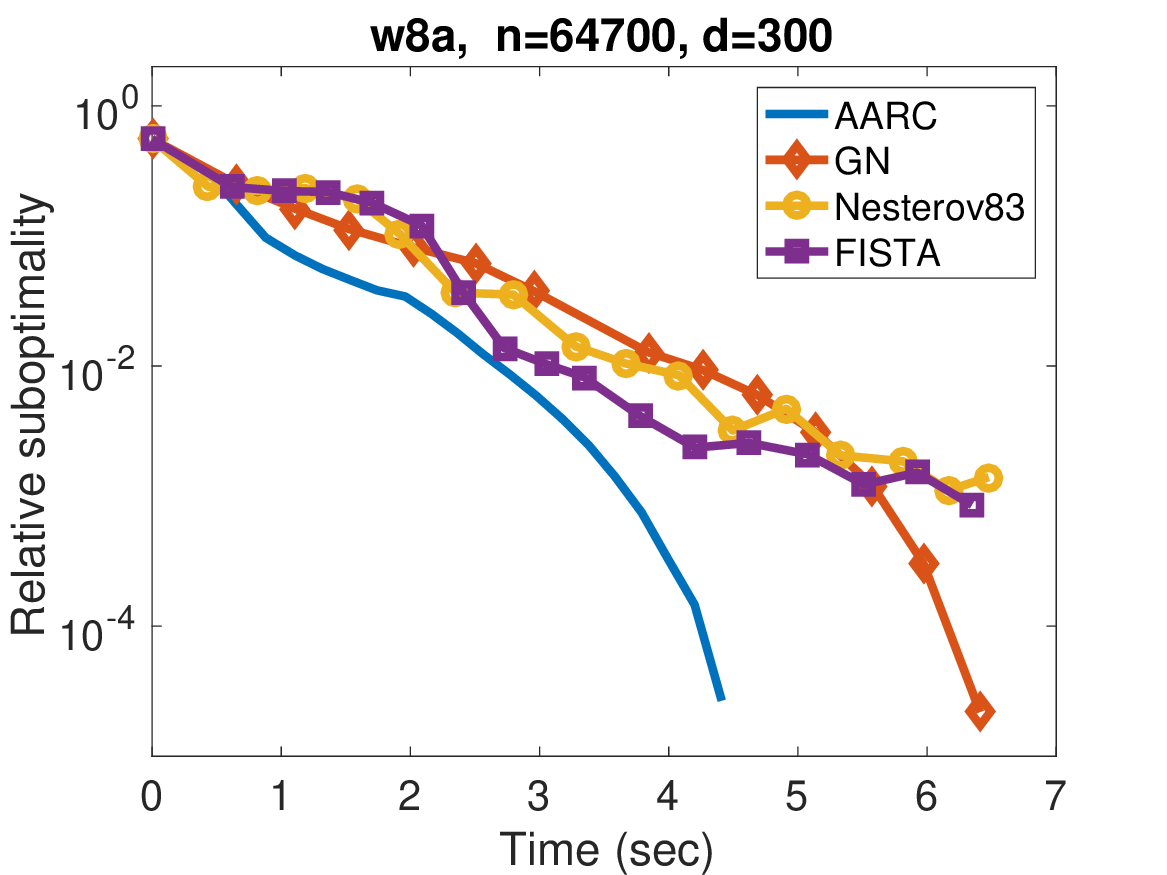}
	\caption{Iteration counts and computational time of the four methods on $\ell_1$-regularized logistic regression.}
	\label{fig3}
	\vspace*{-2em}
\end{figure}

	\section*{Acknowledgement} We would like to express our deep gratitude toward Professor Xi Chen of Stern School of Business at New York University for the fruitful discussions at various stages of this project.

\appendix
\section{Technical Proofs in Section \ref{section:framework}}
First, we bound the total number of iteration in SAS, denoted as $T_1$, and the total number of {iterations} in AAS, denoted as $T_2$.
\begin{lemma} \label{Lemma:Unified-SAS}
	Let
	$ \bar{\sigma}_1= \max\left\{\sigma_0,  \ (p+1) \gamma_2 \beta_p \right\}$, 
	{where $\beta_p$ is defined in \eqref{Def:Effective-Objective-Solution}.} We have
	$ T_1 \leq 1 + \frac{2}{\log\left(\gamma_1\right)}\log\left(\frac{\bar{\sigma}_1}{\sigma_{\min}}\right). $
\end{lemma}
The lemma above is motivated from Theorem 2.1 in \cite{Cartis-2011-Adaptive-II}, and the proof is omitted as it is mostly identical to the one in \cite{Cartis-2011-Adaptive-II}.

\begin{lemma}\label{Lemma:Second-Order-AAS}
	Let $\SCal$ be the set of successful iteration counts in the total iteration count {of AAS}  and
	\[ \bar{\sigma}_2= \max\left\{\bar{\sigma}_1, \ \gamma_2\left(\kappa_{\theta} + \rho_p + \eta\right)\right\}, { -\lambda_{\min}\left(\nabla^2\overline{m}(\bar{\x}_0;{{\x_0}}) \right)/\|\bar{\x}_0 - {\x_0}\|^{p-1}} \} \]
	where $\kappa_\theta$ and $\rho_p$ are defined in \eqref{Def:Effective-Gradient-Solution} and \eqref{Criterion:Approximate-Adaptive} respectively, { $\x_0$ is the initial point of SAS, and $\bar{\x}_0$ is the output of SAS}. Then we have { $\sigma_{\min}\le\sigma_i \le \bar\sigma_2$ for all $i$ in AAS and}
	$ T_2 \leq \left(1+\frac{2}{\log(\gamma_1)}\log\left(\frac{\bar{\sigma}_2}{\sigma_{\min}}\right)\right)|\SCal|. $
\end{lemma}
\begin{proof}
	We observe that
	\begin{eqnarray}\label{Inequality:Lemma-A2}
	& &\left(\y_j - \x_{i+1}\right)^\top \left( \nabla \overline{m}(\x_{i+1};\y_j)+ \sigma_i \left\|\x_{i+1} - \y_j\right\|^{p-1} \left(\x_{i+1} - \y_j\right) + \xi_{i+1}\right) \\
	& \geq &  - \left\| \y_j - \x_{i+1}\right\| \cdot \left\|\nabla \overline{m}(\x_{i+1};\y_j)+ \sigma_i \left\|\x_{i+1} - \y_j\right\|^{p-1} \left(\x_{i+1} - \y_j\right) + \xi_{i+1}\right\| \nonumber \\
	& \overset{~\eqref{Criterion:Approximate-Adaptive}}{\geq} & - \kappa_\theta \left\|\x_{i+1} - \y_j\right\|^{p+1}. \nonumber
	\end{eqnarray}
	Consequently, we conclude that
	\begin{eqnarray*}
		\theta(\x_{i+1}, \y_j, \xi_{i+1}) & = & \frac{\left(\y_j - \x_{i+1}\right)^\top\left(\nabla f(\x_{i+1}) + \xi_{i+1}\right)}{\left\|\y_j - \x_{i+1}\right\|^{p+1}} \\
		&=&{ \frac{ \left(\y_j - \x_{i+1}\right)^\top \left( \nabla f(\x_{i+1}) - \nabla \overline{m}(\x_{i+1};\y_j)- \sigma_i \left\|\x_{i+1} - \y_j\right\|^{p-1} \left(\x_{i+1} - \y_j\right) \right)}{\left\|\y_j - \x_{i+1}\right\|^{p+1}} }\\
		&&{  + \frac{ \left(\y_j - \x_{i+1}\right)^\top \left( \nabla \overline{m}(\x_{i+1};\y_j)+ \sigma_i \left\|\x_{i+1} - \y_j\right\|^{p-1} \left(\x_{i+1} - \y_j\right) + \xi_{i+1}\right)}{\left\|\y_j - \x_{i+1}\right\|^{p+1}} }\\
		& \overset{~\eqref{Inequality:Lemma-A2}}{\geq} & \sigma_i - \kappa_\theta + \frac{\left(\y_j - \x_{i+1}\right)^\top\left(\nabla f(\x_{i+1}) - \nabla \overline{m}(\x_{i+1};\y_j) \right)}{\left\|\y_j - \x_{i+1}\right\|^{p+1}} \\
		& \geq & \sigma_i - \kappa_\theta - \frac{\left\|\y_j - \x_{i+1}\right\|\left\| \nabla f(\x_{i+1}) - \nabla \overline{m}(\x_{i+1};\y_j) \right\|}{\left\|\y_j - \x_{i+1}\right\|^{p+1}}  \\
		& \overset{~\eqref{Def:Effective-Gradient-Solution}}{\geq} & \sigma_i - \kappa_\theta - \frac{\rho_p \left\|\y_j - \x_{i+1}\right\|^{p+1} }{\left\|\y_j - \x_{i+1}\right\|^{p+1}} \\
		& = & \sigma_i - \kappa_\theta - \rho_p,
	\end{eqnarray*}
	and $\sigma_i \geq \kappa_\theta + \rho_p + \eta \ \Longrightarrow \ \theta(\x_{i+1}, \y_j, \xi_{i+1}) \geq \eta$.
	This implies that
	\[
	\sigma_{i+1} \leq \sigma_i \leq \gamma_2 \sigma_{i-1} \leq \gamma_2 \left(\kappa_\theta + \rho_p + \eta\right),\quad \forall \ i \in \SCal.
	\]
	Therefore, $\sigma_i$ can be upper bounded by $\bar{\sigma}_2$ and lower bounded by $\sigma_{\min}$ in AAS. In addition, $\gamma_1 \sigma_i \leq \sigma_{i+1}$ for any $i \notin \SCal$. Therefore, we have
	\[
	\frac{\bar{\sigma}_2}{\sigma_{\min}} \geq \frac{\sigma_{T_2}}{\sigma_0} = \prod_{i\in\SCal} \frac{\sigma_{i+1}}{\sigma_i} \cdot \prod_{i\notin\SCal} \frac{\sigma_{i+1}}{\sigma_i} \geq \gamma_1^{T_2-|\SCal|}\left(\frac{\sigma_{\min}}{\bar{\sigma}_2}\right)^{|\SCal|},
	\]
	which further implies an upper bound for $T_2$, completing the proof.
\end{proof}
Next we proceed to bounding the total number of times updating the regularization parameter $\tau$ in the auxiliary model, which is denoted as $T_3$. This requires three key technical lemmas presented below.
\begin{lemma} {(Lemma 2 in \cite{Nesterov-2008-Accelerating})}	For any $\g \in \br^d$, $\s \in \br^d$ and integer $q \geq 2$, we have
	\begin{equation}\label{Inequality:First-Second-Order-General}
	\g^\top \s + \frac{\sigma\left\|\s\right\|^q}{q} \geq -\frac{q-1}{q} \left(\frac{\left\|\g\right\|^{q}}{\sigma} \right)^{\frac{1}{q-1}}.
	\end{equation}
\end{lemma}

\begin{lemma}\label{Lemma:Unified-Auxillary-Minimizer}
	For the minimizer of $\psi_j(\z, \tau_j)$, i.e.,
	$\z_j = \argmin\limits_{\z\in\br^d} \ \psi_j(\z, \tau_j)$,
	we have
	\[ \psi_j(\z, \tau_j) - \psi_j(\z_j, \tau_j) \geq  {\frac{\tau_j}{2^{p} } \frac{\left\|\z - \z_j\right\|^{p+1}}{p+1}} .\]
\end{lemma}
\begin{proof}
	Recall that $\psi_j(\z, \tau_j)$ is the sum of a linear function and a {$(p+1)$}-th powered regularization function:
	$ \psi_j(\z, \tau_j) = l_j(\z) + \tau_j R(\z) = l_j(\z) + \frac{\tau_j}{2} {\frac{\left\|\z - \x_0\right\|^{p+1}}{p+1}}. $
	Thus, we have
	\begin{eqnarray*}
		& & \psi_j(\z, \tau_j) - \psi_j(\z_j, \tau_j) \\
		& = & \left(\z - \z_j\right)^\top \nabla l_j(\z_j) + \tau_j \left(R(\z) - R(\z_j)\right) \\
		&\ge& \left(\z - \z_j\right)^\top \nabla l_j(\z_j) + \tau_j \left(\z - \z_j\right)^\top \nabla R(\z_j) + { \frac{\tau_j}{2^{p} }\frac{\left\|\z - \z_j\right\|^{p+1}}{p+1} }.
	\end{eqnarray*}
	where the inequality is due to Lemma 4 in \cite{Nesterov-2008-Accelerating}.Since $\z_j$ is the minimizer of $\psi_j(\z, \tau_j)$ over $\z \in \br^d$, we have
	$ \nabla l_j(\z_j) + \tau_j \nabla R(\z_j) = \nabla \psi_j(\z_j, \tau_j) = 0$.
	Combining the above two formulas yields the desired result.
\end{proof}

\begin{lemma}\label{Lemma:Second-Order-Gradient}
	For any $j \geq 0$ in AAS, then we have
	\[ \left\| \nabla f(\bar{\x}_{j+1}) + \bar{\xi}_{j+1}\right\| \leq \left(\rho_p  + \bar \sigma_2 + \kappa_\theta\right) \left\|\bar \x_{j+1} - \y_j\right\|^p.\]
\end{lemma}
\begin{proof}
	We observe that
	\begin{eqnarray}\label{Inequality:Lemma-A5}
	& & \left\|\nabla \overline{m}(\bar \x_{j+1};\y_j) + \bar \xi_{j+1}\right\|  \\
	& \leq &  \left\|\nabla \overline{m}(\bar \x_{j+1};\y_j)+ \sigma_i \left\|\bar \x_{j+1} - \y_j\right\|^{p-1} \left(\bar \x_{j+1} - \y_j\right) + \bar \xi_{j+1}\right\| + \sigma_i \left\| \bar \x_{j+1} - \y_j\right\|^p \nonumber \\
	& \overset{~\eqref{Criterion:Approximate-Adaptive}}{\leq} &  \kappa_\theta \left\|\bar \x_{j+1} - \y_j\right\|^p + \sigma_i \left\| \bar \x_{j+1} - \y_j\right\|^p \nonumber \\
	& \overset{\text{Lemma~\ref{Lemma:Second-Order-AAS}}}{\leq} & \left(\kappa_\theta + \bar{\sigma}_2\right) \left\|\bar \x_{j+1} - \y_j\right\|^p. \nonumber
	\end{eqnarray}
	Therefore,
	\begin{eqnarray*}
		\left\| \nabla f(\bar{\x}_{j+1}) + \bar{\xi}_{j+1}\right\| & \leq  & \left\| \nabla f(\bar{\x}_{j+1}) - \nabla \overline{m}(\x_{j+1};\y_j) \right\| +  \left\| \nabla \overline{m}(\bar \x_{j+1};\y_j) + \bar{\xi}_{j+1} \right\| \\
		& \overset{~\eqref{Def:Effective-Gradient-Solution}}{\leq} & \rho_p \|\bar \x_{j+1} - \y_{j}\|^p + \left(\kappa_\theta + \bar{\sigma}_2\right) \left\|\bar \x_{j+1} - \y_j\right\|^p \\
		& = & \left(\rho_p + \bar{\sigma}_2 + \kappa_\theta\right) \left\|\bar \x_{j+1} - \y_j\right\|^p.
	\end{eqnarray*}
\end{proof}
We remark that the above result is motivated from Lemma 5.2 in \cite{Cartis-2011-Adaptive-II}, which originally works for cubic regularized methods with smooth objective function. Next, we bound $T_3$, the total number of times updating $\tau$ in the auxiliary model:
\begin{lemma}\label{Lemma:AAS-Auxillary}
	For any {successful iteration} $j \geq 0$ in AAS, we have
	\[ \psi_j(\z_j, \tau_j) \geq \frac{\Pi_{\ell=1}^{p+1}(j+\ell)}{(p+1)!} F(\bar{\x}_j) \]
	provided that
	$
	\tau_j \geq \frac{{2^p}\left(\rho_p + \bar{\sigma}_2 + \kappa_{\theta}\right)^{p+1} p^{p-1}}{\eta^p (p-1)!} > 0
	$.
	As a consequence,
	\[
	T_3 \leq 1 + \left\lceil \frac{1}{\log\left(\gamma_3\right)}\log\left(\frac{{2^p}\left(\rho_p + \bar{\sigma}_2 + \kappa_{\theta}\right)^{p+1} p^{p-1}}{\eta^p (p-1)! \tau_0}\right) \right\rceil .\]
\end{lemma}
\begin{proof}
	We shall prove by induction. First of all, the base case of $j=0$ holds true due to the fact that $\psi_0(\z_0, \tau_0)= \min_{\z \in \br^d}F(\bar{\x}_0) + \frac{\tau_0\left\|\z - \bar \x_0\right\|^{p+1}}{2(p+1)} = F(\bar{\x}_0)$.
	Then, we assume the result to hold for some 
	$j = j_0$. It remains to prove the result for the case $j = j_0+1$. By the induction hypothesis, and by Lemma~\ref{Lemma:Unified-Auxillary-Minimizer}, we have
	\begin{eqnarray}\label{Inequality:Induction}
	\psi_{j_0}(\z, \tau_{j_0}) & \geq & \psi_{j_0}(\z_{j_0}, \tau_{j_0}) + \frac{\tau_{j_0}\left\|\z - \z_{j_0}\right\|^{p+1}}{{2^p}(p+1)}\nonumber \\
	& \geq & \frac{\Pi_{\ell=1}^{p+1}(j_0+\ell)}{(p+1)!} F(\bar{\x}_{j_0}) + \frac{\tau_{j_0}\left\|\z - \z_{j_0}\right\|^{p+1}}{{2^p}(p+1)}.
	\end{eqnarray}
	Furthermore, observe that
	\begin{eqnarray}\label{Inequality:AAS-Auxillary}
	&   &	\psi_{j_0+1}(\z_{j_0+1}, \tau_{j_0+1}) \nonumber \\
	& = & \min_{\z\in\br^d} \ \psi_{j_0+1}(\z, \tau_{j_0+1}) \nonumber \\
	& = & \min_{\z\in\br^d} \ \left\{l_{j_0+1}(\z) + \tau_{j_0+1} R(\z)\right\} \nonumber \\
	& = & \min_{\z\in\br^d} \ \left\{l_{j_0}(\z) + \Delta l_{j_0}(\z; \bar{\x}_{j_0+1},\bar \xi_{j_0+1})+ \tau_{j_0} R(\z) + \left(\tau_{j_0+1} - \tau_{j_0}\right) R(\z)\right\} \nonumber \\
	& \geq & \min_{\z\in\br^d} \ \left\{\psi_{j_0}(\z, \tau_{j_0}) + \Delta l_{j_0}(\z;\bar{\x}_{j_0+1},\bar \xi_{j_0+1}\right\},
	\end{eqnarray}
	where the last inequality is because of the fact that $\tau_{j_0+1} \geq \tau_{j_0}$ and $R(\z) \geq 0$ for any $\z \in \br^d$, and
	\begin{eqnarray*}
		\Delta l_{j_0}(\z; \bar{\x}_{j_0+1},  \bar{\xi}_{j_0+1}) & = &
		\frac{\Pi_{\ell = 2}^{p+1}(j_0 + \ell)}{p!} \left[ F(\bar{\x}_{j_0+1}) +\left(\z - \bar{\x}_{j_0+1}\right)^\top\left(\nabla f(\bar{\x}_{j_0+1}) + \bar{\xi}_{j_0+1}\right)\right].
	\end{eqnarray*}
	Therefore, we have
	\begin{eqnarray}\label{Inequality:First-Order-AAS-Auxillary}
	& & \psi_{j_0}(\z, \tau_{j_0}) + \Delta l_{j_0}(\z, \bar{\x}_{j_0+1}) \nonumber \\
	& \overset{~\eqref{Inequality:Induction}}{\geq} & \frac{\Pi_{\ell=1}^{p+1}(j_0+\ell)}{(p+1)!} F(\bar{\x}_{j_0}) + \frac{\tau_{j_0}\left\|\z - \z_{j_0}\right\|^{p+1}}{{2^p}(p+1)} \nonumber \\
	& & +\frac{\Pi_{\ell = 2}^{p+1}(j_0 + \ell)}{p!} \left[ F(\bar{\x}_{j_0+1}) +\left(\z - \bar{\x}_{j_0+1}\right)^\top\left(\nabla f (\bar{\x}_{j_0+1}) + \bar{\xi}_{j_0+1}\right)\right] \nonumber \\
	& \overset{\text{Assumption~\ref{Assumption:Objective-Standard}}}{\geq} & \frac{\Pi_{\ell=1}^{p+1}(j_0+\ell)}{(p+1)!}\left[ F(\bar{\x}_{j_0+1}) +\left(\bar{\x}_{j_0} - \bar{\x}_{j_0+1}\right)^\top\left(\nabla f(\bar{\x}_{j_0+1}) + \bar{\xi}_{j_0+1}\right)\right] \nonumber \\
	& & + \frac{\Pi_{\ell = 2}^{p+1}(j_0 + \ell)}{p!} \left[ F(\bar{\x}_{j_0+1}) +\left(\z - \bar{\x}_{j_0+1}\right)^\top\left(\nabla f(\bar{\x}_{j_0+1}) + \bar{\xi}_{j_0+1}\right)\right]  \nonumber \\
	& & + \frac{\tau_{j_0}\left\|\z - \z_{j_0}\right\|^{p+1}}{{2^p}(p+1)} \nonumber \\
	& = & \frac{\Pi_{\ell=1}^{p+1}(j_0+1+\ell)}{(p+1)!}F(\bar \x_{j_0 +1}) + \frac{\tau_{j_0}\left\|\z - \z_{j_0}\right\|^{p+1}}{{2^p}(p+1)} \nonumber \\
	& & + \frac{\Pi_{\ell=1}^{p+1}(j_0+\ell)}{(p+1)!} \left(\bar{\x}_{j_0}-\bar{\x}_{j_0+1}\right)^\top\left(\nabla f(\bar{\x}_{j_0+1}) + \bar{\xi}_{j_0+1}\right) \nonumber \\
	& & + \frac{\Pi_{\ell = 2}^{p+1}(j_0 + \ell)}{p!} \left(\z - \bar{\x}_{j_0+1}\right)^\top\left(\nabla f(\bar{\x}_{j_0+1}) + \bar{\xi}_{j_0+1}\right),
	\end{eqnarray}
	{where the last equality is due to the fact that \begin{eqnarray*}\frac{\Pi_{\ell=1}^{p+1}(j_0+\ell)}{(p+1)!} +  \frac{\Pi_{\ell = 2}^{p+1}(j_0 + \ell)}{p!} &=&\frac{(j_0+1)\Pi_{\ell=2}^{p+1}(j_0+\ell) + (p+1)\Pi_{\ell = 2}^{p+1}(j_0 + \ell) }{(p+1)!}\\ &=&\frac{\Pi_{\ell=2}^{p+2}(j_0+\ell)}{(p+1)!}=\frac{\Pi_{\ell=1}^{p+1}(j_0+1+\ell)}{(p+1)!}.
	\end{eqnarray*}}
	Moreover, $\y_{j_0}$ in the algorithm is constructed to satisfy
	$ \y_{j_0} = \frac{j_0+1}{j_0+p+2}\bar{\x}_{j_0} + \frac{p+1}{j_0+p+2}\z_{j_0}$,
	and thus
	\begin{eqnarray} \label{Inequality:First-Order-AAS-Auxillary-Iterate}
	\frac{\Pi_{\ell =1}^{p+1}(j_0+\ell)}{(p+1)!} \bar{\x}_{j_0}
	& = & \frac{\Pi_{\ell=1}^{p+1}(j_0+1+\ell)}{(p+1)!} \left(\frac{j_0+1}{j_0+p+2}\bar{\x}_{j_0}\right) \nonumber \\
	& = & \frac{\Pi_{\ell=1}^{p+1}(j_0+1+\ell)}{(p+1)!} \left(\y_{j_0} - \frac{p+1}{j_0+p+2} \z_{j_0}\right) \nonumber \\
	& = & \frac{\Pi_{\ell=1}^{p+1}(j_0+1+\ell)}{(p+1)!} \y_{j_0} - \frac{\Pi_{\ell = 2}^{p+1}(j_0 + \ell)}{p!}\z_{j_0}.
	\end{eqnarray}
	Combining~\eqref{Inequality:AAS-Auxillary},~\eqref{Inequality:First-Order-AAS-Auxillary} and~\eqref{Inequality:First-Order-AAS-Auxillary-Iterate} yields
	\begin{eqnarray*}
		& & \psi_{j_0+1}(\z_{j_0+1}, \tau_{j_0+1}) \\
		& \geq & \min_{\z\in\br^d} \ \left\{\frac{\Pi_{\ell = 2}^{p+1}(j_0 + \ell)}{p!}\left(\z - \z_{j_0}\right)^\top \left(\nabla f(\bar{\x}_{j_0+1}) + \bar{\xi}_{j_0+1}\right) + \frac{\tau_{j_0}\left\|\z - \z_{j_0}\right\|^{p+1}}{{2^p}(p+1)}\right\} \\
		& & + \frac{\Pi_{\ell=1}^{p+1}(j_0+1+\ell)}{(p+1)!} \left(\y_{j_0} - \bar{\x}_{j_0+1}\right)^\top \left(\nabla f(\bar{\x}_{j_0+1}) + \bar{\xi}_{j_0+1}\right) + \frac{\Pi_{\ell=1}^{p+1}(j_0+1+\ell)}{(p+1)!}F(\bar{\x}_{j_0+1}).
	\end{eqnarray*}
	Furthermore, since $j_0$ is a successful iteration, we have
	\begin{eqnarray*}
		\left(\y_{j_0}-\bar{\x}_{j_0+1}\right)^\top\left(\nabla f(\bar{\x}_{j_0+1}) + \bar{\xi}_{j_0+1}\right) & \geq & \eta \left\| \y_{j_0}-\bar{\x}_{j_0+1} \right\|^{p+1} \\
		& \overset{\text{Lemma~\ref{Lemma:Second-Order-Gradient}}}\geq & \eta\left(\frac{\left\| \nabla f(\bar{\x}_{j_0+1}) + \bar{\xi}_{j_0+1}\right\|}{\rho_p + \bar \sigma_2 + \kappa_\theta}\right)^{1+\frac{1}{p}}.
	\end{eqnarray*}
	Thus, it suffices to establish
	\begin{eqnarray}\label{Inequlity:desired}
	\frac{\eta\Pi_{\ell=1}^{p+1}(j_0+1+\ell)}{(p+1)!} \left(\frac{\left\| \nabla f(\bar{\x}_{j_0+1}) + \bar{\xi}_{j_0+1}\right\|}{\rho_p + \bar{\sigma}_2 + \kappa_\theta} \right)^{1+\frac{1}{p}} + \frac{\tau_{j_0}\left\|\z - \z_{j_0}\right\|^{p+1}}{{2^p}(p+1)} & & \nonumber \\
	+ \frac{\Pi_{\ell = 2}^{p+1}(j_0 + \ell)}{p!}\left(\z - \z_{j_0}\right)^\top \left(\nabla f(\bar{\x}_{j_0+1}) + \bar{\xi}_{j_0+1}\right)  & \geq &  0,\; \forall\; \z \in \br^d.
	\end{eqnarray}
	Indeed, applying~\eqref{Inequality:First-Second-Order-General} with
	\[
	{\g }= \frac{\Pi_{\ell = 2}^{p+1}(j_0 + \ell)}{p!} \left(\nabla f(\bar{\x}_{j_0+1}) + \bar{\xi}_{j_0+1}\right), \quad \s = \z-\z_{j_0}, \quad \sigma = \frac{\tau_{j_0}}{{2^p}}, \quad q=p+1
	\]
	we obtain that
	\begin{eqnarray*}
		& & \frac{\Pi_{\ell = 2}^{p+1}(j_0 + \ell)}{p!}\left(\z - \z_{j_0}\right)^\top \left(\nabla f(\bar{\x}_{j_0+1}) + \bar{\xi}_{j_0+1}\right) + \frac{\tau_{j_0}\left\|\z - \z_{j_0}\right\|^{p+1}}{{2^p}(p+1)} \\
		& \geq & -\frac{p}{p+1}\left(\frac{{2^p}}{\tau_{j_0}}\right)^{\frac{1}{p}}\left(\frac{\Pi_{\ell = 2}^{p+1}(j_0 + \ell)}{p!} \left\|\nabla f(\bar{\x}_{j_0+1}) + \bar{\xi}_{j_0+1}\right\|\right)^{1+\frac{1}{p}}.
	\end{eqnarray*}
	Therefore,~\eqref{Inequlity:desired} is equivalent to
	\begin{eqnarray*}
		\tau_{j_0} & \geq & \frac{{2^p}\left(\rho_p + \bar{\sigma}_2 + \kappa_{\theta}\right)^{p+1}}{\eta^p} \left(\frac{p}{p+1} \right)^p\left(\frac{\Pi_{\ell = 2}^{p+1}(j_0 + \ell)}{p!}\right)^{p+1}\left(\frac{(p+1)!}{\Pi_{\ell=1}^{p+1}(j_0+1+\ell)}\right)^p \\
		&= & {\frac{2^p\left(\rho_p + \bar{\sigma}_2 + \kappa_{\theta}\right)^{p+1}}{\eta^p}\left(\Pi_{\ell = 2}^{p+1}(j_0 + \ell)\right) \left(\frac{\Pi_{\ell = 2}^{p+1}(j_0 + \ell)}{\Pi_{\ell=1}^{p+1}(j_0+1+\ell)}\right)^{p}\frac{p^p}{p!}}\\
		& = & \frac{{2^p}\left(\rho_p + \bar{\sigma}_2 + \kappa_{\theta}\right)^{p+1}}{\eta^p}\left(\frac{\Pi_{\ell = 2}^{p+1}(j_0 + \ell)}{(j_0+p+2)^p}\right)\frac{p^p}{p!}.
	\end{eqnarray*}
	Now, observe that the {first part of the} conclusion would follow if
	\[
	\tau_{j_0} \ge \frac{{2^p}\left(\rho_p + \bar{\sigma}_2 + \kappa_{\theta}\right)^{p+1}}{\eta^p}\frac{p^{p-1}}{(p-1)!}
	\]
	holds, which is the condition of the lemma.
	
	{To prove the remaining part of the conclusion, we note that $\tau_j$ can only be updated in the successful iteration of AAS and it increases by a factor of $\gamma_3$ when updated. Recall that $T_3$ is the total number of updating counts for $\tau_j$. Then according to the first part of the conclusion, $\tau_j$ will not be updated once
		$$
		\tau_0 \gamma_3^{T_3} \ge \frac{2^p\left(\rho_p + \bar{\sigma}_2 + \kappa_{\theta}\right)^{p+1}}{\eta^p}\frac{p^{p-1}}{(p-1)!},
		$$
		which means that $T_3$ is the least integer that makes the inequality above hold and thus the conclusion follows.}
\end{proof}
Now, we analyze the initial iterate in AAS, which is also the re-initialized iterate returned by SAS.
\begin{theorem} \label{Thm:base-case}
	Let $\bar{\x}_0$ be the initial iterate in AAS of Algorithm \ref{Algorithm:Framework}, then by letting $\hat{\sigma}_1:=\max\left\{\bar{\sigma}_1, \frac{L_p}{(p-1)!} \right\}$ we have that
	\begin{eqnarray*}
		F(\bar{\x}_0) \le \psi_0(\z, \tau_0) 
		&\leq& F(\z) + \frac{(p+1)\kappa_p + {\hat{\sigma}_1}}{p+1}\left\|\z-\x_0\right\|^{p+1}  \\
		& &+ \bar{\kappa}_p \left\|\z-\x_0\right\|^p +\frac{\tau_0\left\|\z - \bar \x_0\right\|^{p+1}}{2(p+1)} {+ (\kappa_\theta + \hat{\sigma}_1) (2D)^{p+1}},
	\end{eqnarray*}
\end{theorem}
\begin{proof}
	Recall that $F(\bar{\x}_0) = \min_{\z \in \br^d} \ \left\{F(\bar{\x}_0) + \tau_0 R(\z) \right\} = \psi_0(\z_0, \tau_0)$. It suffices to show the inequality on the right hand side. Denote $\x_0 \in \br^d$ to be the initial iterate of SAS, {$\sigma^{SAS}$ be the regularized parameter associated with $\bar{\x}_0$}, and $\bar{\x}_0^m \in \br^d$ to be the global minimizer of $m(\x;\x_0, \sigma^{SAS})$ over $\br^d$. Since  $\bar{\x}_0$ is also the output returned by SAS, it holds that $m(\bar{\x}_0^m;\x_0, \sigma^{SAS})  \leq  m(\bar{\x}_0;\x_0, \sigma^{SAS}) \le m({\x}_0;\x_0, \sigma^{SAS}) = F({\x}_0)$. Moreover, {the updating rule of $\sigma_i$ in SAS implies to that $\sigma^{SAS} \ge \sigma_{\min}$}, which further indicates $m(\x;\x_0, \sigma_{\min}) \le m(\x;\x_0,  \sigma^{SAS}) $ for all $\x$ and thus $\mathcal{L}(\x_0,\sigma^{SAS}) \subseteq \mathcal{L}(\x_0,\bar \sigma_{\min})$. Then according to \eqref{Bounded-levelset},
	\begin{equation}\label{Bounded-levelset-x}
	\| \bar \x_0 - \x^* \| \le D \quad \mbox{and} \quad \| \bar \x_0^m - \x^* \| \le D.
	\end{equation}
	{ If $\overline{m}(\y;\x)$ is convex then ${m}(\y;\x,\sigma)$ is convex as well. Moreover, as we mentioned earlier, ${m}(\y;\x,\sigma)$ is { not} necessarily convex for high-order adaptive accelerating method.
		In this case, let $\sigma_0^{AAS}=\max\{\sigma^{SAS},  -\lambda_{\min}\left(\nabla^2\overline{m}(\bar{\x}_0;{\x_0}) \right)/\|\bar{\x}_0 - \x_0\|^{p-1}\}$. According to the discussion above \eqref{sigma-0-AAS}, $m(\y;\x_0, \sigma_0^{AAS})$ is convex at $\bar{\x}_0$ and
		\[
		F(\bar{\x}_0)  \leq  m(\bar{\x}_0;\x_0, \sigma_0^{AAS}) =  m(\bar{\x}_0;\x_0, \sigma_0^{AAS}) - m(\bar{\x}^m_0; \x_0,  \sigma_0^{AAS}) + m(\bar{\x}^m_0; \x_0,  \sigma_0^{AAS}).
		\]
		Moreover, combining equality (2.3) in \cite{Nesterov-2018} and \eqref{Hessian-p-power} yields
		$$\nabla^2 f(\y) \preceq \nabla^2 \overline{m}(\y; \x) +  \frac{L_p \|\y - \x\|^{p-1}}{(p-1)!}  I \preceq \nabla^2 \overline{m}(\y; \x) + \sigma \|\y - \x\|^{p-1} I  $$
		when $\sigma \ge \frac{L_p}{(p-1)!}$ and ${m}(\y; \x,\sigma)$ is a convex function for any $\x$.
		Therefore,
		\begin{equation}\label{upper-bound-sigma0-AAS}
		\sigma_0^{AAS} \le \hat{\sigma}_1=\max\left\{\bar{\sigma}_1, \frac{L_p}{(p-1)!} \right\},
		\end{equation}
		and there exists some $\bar{\xi}_0 \in \partial r(\bar{\x}_0)$ (e.g.\ $\bar{\xi}_0$ could be the one that validates \eqref{Criterion:Approximate-Adaptive}) such that
		\begin{eqnarray*}
			& & m(\bar{\x}_0;\x_0, \sigma_0^{AAS}) - m(\bar{\x}^m_0; \x_0,  \sigma_0^{AAS}) \\
			& \leq & - (\nabla \overline{m}(\bar\x_{0};\x_0)+ \sigma_0^{AAS} \left\|\bar\x_{0} - \x_0\right\|^{p-1} \left(\bar \x_{0} - \x_0\right) + \bar\xi_{0})^{\top}(\bar \x_0^{m} - \bar\x_0)\\
			& \leq & \left\|\nabla \overline{m}(\bar\x_{0};\x_0)+ \sigma^{SAS} \left\|\bar\x_{0} - \x_0\right\|^{p-1} \left(\bar \x_{0} - \x_0\right) + \bar\xi_{0}\right\| \left\|\bar\x_0 - \bar \x_0^{m}\right\| \\
			&& + \left\| (\sigma_0^{AAS} - \sigma^{SAS})\left\|\bar\x_{0} - \x_0\right\|^{p-1} \left(\bar \x_{0} - \x_0\right)\right\| \cdot \left\|\bar\x_0 - \bar \x_0^{m}\right\| \\ &\overset{~\eqref{Criterion:Approximate-Adaptive},\eqref{upper-bound-sigma0-AAS}}{\leq}& (\kappa_\theta + \hat{\sigma}_1) \left\| \bar{\x}_0 - \x_0 \right\|^p \left\| \bar{\x}_0 - \bar{\x}^m_0 \right\|.
		\end{eqnarray*}
	}
	On the other hand,
	\begin{eqnarray*}
		m\left(\bar{\x}^m_0; \x_0, \sigma_0^{AAS}\right) & = & \overline{m}(\bar \x_0^m ; \x_0) + \frac{\sigma_0^{AAS}\left\|\bar{\x}^m_0 - \x_0\right\|^{p+1}}{p+1} + r(\bar{\x}^m_0)   \\
		& \leq & \overline{m}(\z ; \x_0) + \frac{\sigma_0^{AAS}\left\|\z - \x_0\right\|^{p+1}}{p+1} + r(\z) \\
		& \overset{\eqref{Def:Effective-Objective-All}}{\leq} & f(\z) + \kappa_{p}\left\| \z - \x_0 \right\|^{p+1} + \bar{\kappa}_p \| \z - \x_0 \|^p + \frac{\sigma_0^{AAS} \left\|\z - \x_0 \right\|^{p+1}}{p+1}  + r(\z) \\
		& \overset{\eqref{upper-bound-sigma0-AAS}}{\leq} & F(\z) + \frac{(p+1)\kappa_p + {\hat{\sigma}_1}}{p+1}\left\|\z-\x_0\right\|^{p+1} + \bar{\kappa}_p \left\|\z-\x_0\right\|^p .
	\end{eqnarray*}
	Combining the above two inequalities, we have
	\begin{eqnarray*}
		\psi_0(\z, \tau_0) & = & F(\bar{\x}_0) + \frac{\tau_0\left\|\z - \bar{\x}_0\right\|^{p+1}}{2(p+1)}\\
		&\leq & \left( m(\bar{\x}_0;\x_0, \sigma_0^{AAS}) - m(\bar{\x}^m_0; \x_0,  \sigma_0^{AAS})\right) + m(\bar{\x}^m_0; \x_0,  \sigma_0^{AAS}) + \frac{\tau_0\left\|\z - \bar \x_0\right\|^{p+1}}{2(p+1)} \\
		&\leq & F(\z) + \frac{(p+1)\kappa_p + {\hat{\sigma}_1}}{p+1} \left\|\z-\x_0\right\|^{p+1} + \bar{\kappa}_p \left\|\z-\x_0\right\|^p +\frac{\tau_0\left\|\z - \bar \x_0\right\|^{p+1}}{2(p+1)} \\
		& & + (\kappa_\theta+{\hat{\sigma}_1}) \left\| \bar{\x}_0 - \x_0 \right\|^p \left\|\bar{\x}_0 - \bar{\x}^m_0\right\| \\
		& \leq & F(\z) + \frac{(p+1)\kappa_p + {\hat{\sigma}_1}}{p+1} \left\|\z-\x_0\right\|^{p+1} + \bar{\kappa}_p \left\|\z-\x_0\right\|^p +\frac{\tau_0\left\|\z - \bar \x_0\right\|^{p+1}}{2(p+1)} \\
		& & + (\kappa_\theta + {\hat{\sigma}_1})(2D)^{p+1},
	\end{eqnarray*}
	{where} the last inequality is due to \eqref{Bounded-levelset-x} and \eqref{Bounded-levelset}.
\end{proof}
Next, we proceed to analyzing all the iterates in AAS.
\begin{theorem}\label{Thm:Func-Val-Bound}
	The sequence $\{\bar{\x}_j, \ j\geq 0\}$  generated by \textsf{AAS} in \textsf{UAA} satisfies
	\begin{eqnarray} \label{KeyInquality:F}
	& & \frac{\Pi_{\ell=1}^{p+1}(j+\ell)}{(p+1)!} F(\bar{\x}_j) \ \leq \  \psi_{j}(\z, \tau_{j}) 
	\leq  \frac{\Pi_{\ell=1}^{p+1}(j+\ell)}{(p+1)!} F(\z)
	+  \frac{(p+1)\kappa_p +{\hat{\sigma}_1}}{p+1} \left\|\z-\x_0\right\|^{p+1} \nonumber \\
	& & \quad \quad \quad \quad + \bar{\kappa}_p \left\|\z-\x_0\right\|^p  + \frac{{ \tau_j}\left\|\z - \bar \x_0\right\|^{p+1}}{2(p+1)}+ (\kappa_\theta + {\hat{\sigma}_1}) (2D)^{p+1}.
	\end{eqnarray}
\end{theorem}
\begin{proof} By the way in which $\psi_{j+1}(\z_{j+1}, \tau_{j+1})$ is updated in AAS, we have
	\[
	\frac{\Pi_{\ell=1}^{p+1}(j+1+ \ell)}{(p+1)!} F(\bar{\x}_{j+1}) \ \leq \ \psi_j(\z_{j+1}, \tau_{j+1}) \ \leq \ \psi_j(\z, \tau_{j+1}), \quad \forall \; j \geq 0.
	\]
	It thus suffices to show the inequality on the right hand side by induction. The base case of $j=0$ has already been proved in Theorem \ref{Thm:base-case}. We now assume the result holds for some $j = j_0$. For the case $j = j_0+1$, indeed we have
	\begin{eqnarray*}
		&& \psi_{j_0+1}(\z_{j_0+1}, \tau_{j_0+1}) \\
		& \leq & \psi_{j_0+1}(\z, \tau_{j_0+1})  \\
		& = & l_{j_0}(\z) + \Delta l_{j_0}(\z; \bar{\x}_{j_0+1}, \bar{\xi}_{j_0+1}) + \frac{\tau_{j_0+1}\left\|\z-\x_0\right\|^{p+1}}{2(p+1)} \\
		& = & \psi_{j_0}(\z, \tau_{j_0}) + \Delta l_{j_0}(\z; \bar{\x}_{j_0+1}, \bar{\xi}_{j_0+1}) + \frac{\left(\tau_{j_0+1} - \tau_{j_0}\right)\left\|\z-\x_0\right\|^{p+1}}{2(p+1)} \\
		& \leq & \frac{\Pi_{\ell=1}^{p+1}(j_0+\ell)}{(p+1)!} F(\z)  +   \frac{(p+1)\kappa_p +  {\hat{\sigma}_1}}{p+1} \left\|\z-\x_0\right\|^{p+1} + \, \bar{\kappa}_p \left\|\z-\x_0\right\|^p + \frac{{\tau_{j_0}}\left\|\z - \bar \x_0\right\|^{p+1}}{2(p+1)} \\
		& & + (\kappa_\theta + {\hat{\sigma}_1}) (2D)^{p+1} + \frac{\Pi_{\ell=2}^{p+1}(j_0+\ell)}{p!}\left[ F(\bar{\x}_{j_0+1}) + \left(\z - \bar{\x}_{j_0+1}\right)^\top \left({\nabla f}(\bar{\x}_{j_0+1}) + \bar{\xi}_{j_0+1} \right)\right]  \\
		& \overset{\text{Assumption~\ref{Assumption:Objective-Standard}}}{\leq} & \frac{\Pi_{\ell=1}^{p+1}(j_0+1+\ell)}{(p+1)!} F(\z)  +   \frac{(p+1)\kappa_p + {\hat{\sigma}_1}}{p+1} \left\|\z-\x_0\right\|^{p+1} + \, \bar{\kappa}_p \left\|\z-\x_0\right\|^p \\
		& & + \frac{{\tau_{j_0}}\left\|\z - \bar \x_0\right\|^{p+1}}{2(p+1)}+ (\kappa_\theta + {\hat{\sigma}_1}) (2D)^{p+1},
	\end{eqnarray*}
	where the second last inequality is due to the mathematical induction and $\tau_j$ is monotonically increasing. This completes the proof.
\end{proof}
{\bf Proof of Theorem \ref{Thm:iteration-complexity-uaa}:}
Recall that in the proof of Theorem \ref{Thm:base-case}, we have shown $\|\x^* - \bar \x_0\| \le D$. Then,
taking $\z=\x^*$ in \eqref{KeyInquality:F} yields that
\begin{eqnarray*}
	& & \frac{\Pi_{\ell=1}^{p+1}(j+\ell)}{(p+1)!} F(\bar{\x}_j) \\
	&\leq& \frac{\Pi_{\ell=1}^{p+1}(j+\ell)}{(p+1)!} F(\x^*) +  \frac{(p+1)\kappa_p + {\hat{\sigma}_1}}{p+1}\left\|\x^*-\x_0\right\|^{p+1} + \, \bar{\kappa}_p \left\|\x^*-\x_0\right\|^p  \nonumber \\
	&&+ \frac{{\tau_{j}}\left\|\x^* - \bar \x_0\right\|^{p+1}}{2(p+1)}+  (\kappa_\theta + {\hat{\sigma}_1}) (2D)^{p+1} \nonumber \\
	&\le& \frac{\Pi_{\ell=1}^{p+1}(j+\ell)}{(p+1)!} F(\x^*) +  \frac{(p+1)\kappa_p + {\hat{\sigma}_1}}{p+1}D^{p+1} + \, \bar{\kappa}_p D^p + \frac{{\tau_{j}} D^{p+1}}{2(p+1)} + (\kappa_\theta + {\hat{\sigma}_1}) (2D)^{p+1}.
\end{eqnarray*}
{According to Lemma \ref{Lemma:AAS-Auxillary}, $	\tau_j $ will not be increased once it exceeds $ \frac{2^p\left(\rho_p + \bar{\sigma}_2 + \kappa_{\theta}\right)^{p+1} p^{p-1}}{\eta^p (p-1)!}$. Therefore, we have $$\tau_j \le \max\left\{ \tau_0, \frac{2^p\gamma_3\left(\rho_p + \bar{\sigma}_2 + \kappa_{\theta}\right)^{p+1} p^{p-1}}{\eta^p (p-1)!} \right\} = \hat \sigma_2.$$
	Combining the two equalities above, it holds that
}
\[
F(\bar{\x}_j) - F(\x^*) \leq \frac{(p+1)! \left( \frac{2(p+1)\kappa_p + 2{\hat{\sigma}_1}+{\hat{\sigma}_2}}{2(p+1)}D^{p+1} + \, \bar{\kappa}_p D^p  + (\kappa_\theta + {\hat{\sigma}_1}) (2D)^{p+1} \right)}{\Pi_{\ell=1}^{p+1}(j+\ell)} .
\]
Combining this inequality with Lemmas~\ref{Lemma:Unified-SAS},~\ref{Lemma:Second-Order-AAS} and~\ref{Lemma:AAS-Auxillary} implies the conclusion.
\hfill $\Box$ \vskip 0.1cm
\end{document}